\documentclass[a4paper,12pt]{amsart}
\usepackage[T1]{fontenc}
\usepackage[cp1250]{inputenc}
\usepackage{xcolor}

\usepackage[arrow, matrix, curve]{xy}
\usepackage{amsmath,amsthm,amssymb,dsfont}
\usepackage{amscd,graphics}
\usepackage[dvips]{graphicx}

\newenvironment{Proof of}[1]{\textbf{Proof #1.}}{$\qquad \blacksquare$\par}

\DeclareMathOperator{\Aut}{Aut}

\DeclareMathOperator{\clsp}{\overline{span}}
\DeclareMathOperator{\spane}{span}

\DeclareMathOperator{\Inv}{Inv}
\DeclareMathOperator{\Erg}{Erg}

\DeclareMathOperator{\ess}{ess}

\newcommand{\B}{\mathcal B}
\newcommand{\LL}{\mathcal L}
\newcommand{\RR}{\mathcal R}
\newcommand{\M}{\mathcal M}
\newcommand{\TT}{\mathcal T}

\newcommand{\OO}{\mathcal{O}}

\newcommand{\al}{\alpha}
\newcommand{\K}{\mathcal{K}}

\newcommand{\G}{\mathcal G}
\newcommand{\C}{\mathbb C}
\newcommand{\R}{\mathbb R}
\newcommand{\Z}{\mathbb Z}
\newcommand{\N}{\mathbb N}
\newcommand{\T}{\mathbb T}

\newcommand{\supp}{\textrm{supp}\,}
\newtheorem{thm}{Theorem}[section]
\newtheorem{lem}[thm]{Lemma}

\newtheorem{prop}[thm]{Proposition}
\newtheorem{cor}[thm]{Corollary}
\theoremstyle{definition}
\newtheorem{defn}[thm]{Definition}
\newtheorem{ex}[thm]{Example}

\newtheorem{rem}[thm]{Remark}

\title[Spectrum of weighted isometries]{Spectrum 
of weighted isometries: $C^*$-algebras, transfer operators and topological pressure}

\author{Krzysztof Bardadyn}
\email{kbardadyn@math.uwb.edu.pl}
 \address{Faculty of Mathematics\\
   University  of Bia\l ystok\\
   ul.\@ K.~Cio\l kowskiego 1M\\
   15-245 Bia\l ystok\\
   Poland} 

\author{Bartosz Kosma Kwa\'sniewski}
\email{bartoszk@math.uwb.edu.pl}
 \address{Faculty of Mathematics\\
   University  of Bia\l ystok\\
   ul.\@ K.~Cio\l kowskiego 1M\\
   15-245 Bia\l ystok\\
   Poland}
\subjclass[2010]{47A10, 37A55}
\keywords{spectrum, weighted composition operator, $C^*$-algebra, transfer operator, topological pressure}

\thanks{We  thank Andrei Lebedev for numerous discussions on variational formulas appearing in the text. This work was  supported by the National Science Centre, Poland,  grant number  2019/35/B/ST1/02684} 
 
\begin{document}
\begin{abstract} 
We study the spectrum of  operators $aT\in \B(H)$ on a Hilbert space $H$ where $T$ is an isometry and  $a$ belongs to a commutative $C^*$-subalgebra $C(X)\cong A\subseteq \B(H)$ 
such that the formula $L(a)=T^*aT$ defines a faithful transfer operator on $A$.   Based on the analysis of the $C^*$-algebra $C^*(A,T)$ generated by the operators 
$aT$, $a\in A$, we give  dynamical conditions implying 
that the spectrum $\sigma(aT)$ is invariant under rotation around zero, $\sigma(aT)$ coincides with the essential spectrum $\sigma_{ess}(aT)$ or that $\sigma(aT)$ is the disc $\{z\in \C: |z|\leq r(aT)\}$. 

We get the best results when the underlying mapping $\varphi:X\to X$ is expanding and open. We  prove  for any such map  and a continuous map $c:X\to [0,\infty)$  that the  spectral logarythm of a Ruelle-Perron-Frobenious operator $\mathcal{L}_cf(y)=\sum_{x\in \varphi^{-1}(y)} c(x)f(x)$ is equal to the topological pressure  $P(\ln c,\varphi)$. This extends Ruelle's classical result and
implies the variational principle for the spectral radius:
$$
r(aT)=\max_{\mu\in \Erg(X,\varphi)} \exp\left(\int_{{X}}\ln(\vert a\vert\sqrt{\varrho })\,d\mu
+\frac{h_{\varphi}(\mu)}{2} \right),
$$
where $\Erg(X,\varphi)$ is the set of ergodic Borel probability measures, $h_{\varphi}(\mu)$ is the Kol\-mo\-gorov-Sinai entropy,
and $\varrho:X\to [0,1]$ is the cocycle associated to $L$. In particular, we clarify the relationship between 
the Kolmogorov-Sinai entropy and $t$-entropy introduced by Antonevich, Bakhtin and Lebedev.

\end{abstract}
   \maketitle


 \section{Introduction }

The classical Coburn's theorem \cite{Coburn} states that every non-invertible isometry $T\in \B(H)$ on a Hilbert space $H$ generates a $C^*$-algebra $C^*(T)$ canonically 
isomorphic to the Toeplitz algebra $\mathcal{T}$. Combining this with the well known fact that   $\mathcal{T}$ is generated by the unilateral shift whose spectrum is the unit disk, we infer that 
$\sigma(T)=\{z\in \C:|z|\leq 1\}$ for every non-invertible isometry $T$. Here we  use that the spectrum of an operator in a Hilbert space depends only on the smallest unital 
$C^*$-algebra that contains it. Cuntz \cite{Cu77} proved a higher dimensional version of Coburn's  result, that has immediate spectral consequences for 
certain weighted isometries:
\begin{ex}\label{ex:Cuntz}
Consider the shift map $ \varphi (\xi_1,\xi_2,...):=(\xi_2,\xi_3,...)$  
 on 
the Cantor space $
\Sigma:=\{1,...,n\}^\N$, $n>1$, and let $\mu$ be  the Bernoulli measure - the product measure on $\Sigma$ of uniform distributions on $\{1,...,n\}$. 
Let $H=L^2_\mu(\Sigma)$. Then the composition operator $Th=h\circ \varphi$, $h\in H$, is an isometry.
Let $A\cong C(\Sigma)$ be the algebra of multiplication operators by continuous functions: $(ah)(x)=a(x)h(x)$.  
Then the $C^*$-algebra generated by weighted composition operators $aT$, $a\in A$, is isomorphic to the Cuntz algebra $\mathcal{O}_n$:
$$C^*(aT:a\in A)=C^*(A,T)\cong \mathcal{O}_n.
$$
Indeed,  the operators  $s_i:=\sqrt{n}\cdot \mathds{1}_{\{(i,\xi_2,\xi_3,...)\}}T$, $i=1,...,n$, are  isometries such that $\sum_{i=1}^n s_i s_i^*=1$
and $C^*(A,T)=C^*(s_1,...,s_n)$ (e.g. we have $T=n^{-1/2}\sum_{i=1}^n s_i$). Cuntz proved in \cite{Cu77} that every $C^*$-algebra generated by $n$ isometries with range projections summing up  to $1$ 
is simple and isomorphic to $\mathcal{O}_n$. 
Since $C^*(A,T)\cong \mathcal{O}_n$ is simple and infinite it can not contain any compact operators. Thus the spectrum $\sigma(aT)$, for any $a\in A$,
coincides with the \emph{essential spectrum} $\sigma_{ess}(aT):=\{\lambda\in \C: aT -\lambda 1 \text{ is not Fredholm}\}$. 
Moreover, $\mathcal{O}_n$
is equipped with a gauge circle action that scales $T$ and fixes elements in $A$.   
This implies that $\sigma(aT)$ has a circular symmetry. 
As we explain in the paper, by an analysis of Riesz projectors and using the fact that $A\subseteq \mathcal{O}_n$ is a Cartan $C^*$-subalgebra \cite{Re}, one can prove that
$\sigma(aT)$ is connected and hence a disk. In fact we have
$$
\sigma(aT)=\sigma_{ess}(aT)=\left\{z\in \C:|z|\leq \max_{\mu\in \Erg(\Sigma,\varphi)} \exp \left(
\int_{\Sigma}\ln\frac{\vert a\vert}{\sqrt{n}}\,d\mu
 + \frac{h_{\varphi}(\mu)}{2} \right)\right\}
$$
where the maximum is taken over the ergodic  Borel probability measures and $h_{\varphi}(\mu)$ is the Kol\-mo\-gorov-Sinai entropy. The formula for  $r(aT)$ follows from Ruelle's thermodynamical formalism that relates the 
spectral radius of the Ruelle-Perron-Frobenious operator with topological pressure \cite{Ruelle0}, \cite{Walt1}, \cite{LatStep}, \cite{Ruelle}.
\end{ex}

The aim of the present paper is to develop a general theory that  describes
the spectrum of \emph{weighted isometries} $aT$, $a\in A\subseteq B(H)$, by the analysis similar to that sketched in Example \ref{ex:Cuntz}.
In the case that $T$ is a \emph{unitary} such an analysis was carried out successfully in \cite{O'Donov}, \cite{Lebedev}, 
\cite{Anton_Lebed}. More specifically, 
let $T\in B(H)$ be a unitary on a Hilbert space $H$, and let   $A\subseteq  \B(H)$ be a unital commutative $C^*$-subalgebra such that 
\begin{equation}\label{eq:automorphism_relations}
TAT^*=A.
\end{equation}
Then $\alpha(a):=TaT^*$, $a\in A$, 
is an automorphism of $A$, and identifying $A\cong C(X)$ with the algebra of continuous functions on the Gelfand spectrum $X$ of $A$, $\alpha$ is a composition operator  with a homeomorphism $\varphi:X\to X$. This motivated the authors of \cite{Anton_Lebed} to call operators of the form
$aT$, $a\in A$, \emph{abstract weighted shifts} associated with the automorphism $\alpha:A\to A$.  By \cite[Theorem 1.2.1]{O'Donov}, if $\varphi:X\to X$ is \emph{topologically free},
i.e. the set of periodic points has empty interior, then  $C^*(A,T)\cong A\rtimes_\alpha \Z$ where 
$A\rtimes_\alpha \Z$ is the crossed product of $A$ by $\Z$ action $\{\alpha^{n}\}_{n\in \Z}$. This  is a strong tool
that allows us to study the spectrum $\sigma(aT)$. Combining it with the  variational formula for the spectral radius 
\begin{equation}\label{eq:spectral_radius_invertible}
r(aT)=\max_{\mu\in \Erg(X,\varphi)} \, \exp \int_{X} \ln|a|\, d\mu, 
\end{equation}  
proved  independently by  Kitover \cite{Kitover} and Lebedev \cite{Lebedev}, see also 
\cite{kwa-leb}, one can obtain the following description of $\sigma(aT)$:

\begin{thm}[\cite{Lebedev}, \cite{Anton_Lebed}]\label{thm:lebedev}
Let $aT$, $a\in A\cong C(X)$, be abstract weighted shifts associated with an automorphism $\alpha:A\to A$. 
Suppose that the dual homeomorphism $\varphi:X\to X$ is topologically free. 
For every $a\in A$ the spectrum $\sigma(aT)$ is a union of annuli $\bigcup_{r \in R}\{z\in \C: r_{-}\leq |z| \leq  r_+\}$ and if $a$ is not invertible also a 
disk $\{z\in\C: |z|\leq r_0\}$. Moreover, all Riesz projectors for $\sigma(aT)$ belong to $A$ and induce  a decomposition  $X=X_0\cup \bigcup_{r\in R} X_{r}$  into disjoint closed $\varphi$-invariant sets such that 
$$
r_{-}=\min_{\mu\in \Erg(X_r,\varphi)} \, \exp \int_{X_r} \ln|a|\, d\mu, 
\qquad
r_+=\max_{\mu\in \Erg(X_r,\varphi)} \, \exp \int_{X_r} \ln|a|\, d\mu, 
 $$
for  $r \in R$,
and $r_0=\max_{\mu\in \Inv(X_0,\varphi)}  \exp\left(\int_{{X_0}}\ln\vert a\vert\,d\mu
\right), 
$ if $X_0$ is non-empty (that is, if $a$ is not invertible). 
So if $X$ does not contain non-trivial clopen $\varphi$-invariant sets, then  $\sigma(aT)$ is connected. 
If $A$ does not contain non-zero compact operators, which is automatic if $X$ has no isolated points, then $\sigma(aT)=\sigma_{ess}(aT)$.
\end{thm}
Theorem \ref{thm:lebedev} explains  some operator theoretical phenomena observed  by a number
of authors studying various special cases of weighted composition operators, see \cite[Notes and Remarks on page 154]{Anton_Lebed}.
Similar $C^*$-algebraic arguments apply to other spectral properties of more general non-local operators and have remarkable consequences for 
functional differential equations with \emph{reversible shifts}, see for instance,   \cite{Anton_Lebed},  \cite{Anton_Belousov_Lebed2}, \cite{BFK}.   
The natural need to develop such a theory for \emph{irreversible shifts}  has so far been blocked by the lack of appropriate 
$C^*$-algebraic tools.
In the present paper we wish to provide a first meaningful step towards such a generalization.
This is non-trivial and only possible because of  recent advances in the theory of $C^*$-algebras. 
We hope this will renew interest of operator theorists in the modern $C^*$-algebra theory, as well as strengthen links between these fields  and the thermodynamical formalism.

There are essentially two ways in which one can generalize  abstract weighted shifts to the irreversible case. If $T$ is a unitary, then  \eqref{eq:automorphism_relations} is equivalent to two inclusions $TAT^*\subseteq A$ and 
$T^*AT\subseteq A$.  If $T$ is a non-unitary isometry, one can not expect \eqref{eq:automorphism_relations} to hold. 
One natural choice is to assume that $TAT^*\subseteq A$, as then $\alpha(a)=TaT^*$ defines an endomorphism of $A$ (in essence this boils down  to relations in \cite[Remark 4.6]{Anton_Lebed}). 
The corresponding analysis is carried out in \cite{kwa-phd}, \cite{kwa-leb}, and the associated $C^*$-algebras $C^*(A,T)$ 
are modeled by crossed products by endomorphisms \cite{kwa-leb0}, \cite{kwa_Exel}. 
However this approach does not cover Example \ref{ex:Cuntz} and is somewhat easier than what we need. In the present paper we consider the situation where 
 $T\in B(H)$ is an isometry   and   $A\subseteq  \B(H)$ is a unital commutative $C^*$-subalgebra such that 
\begin{itemize}
\item[(A1)] $
T^*AT\subseteq A$;
\item[(A2)] there is a unital endomorphism $\alpha:A\to A$ such that 
$
Ta=\alpha(a)T$,  for all $a\in A$.
\end{itemize}
Then $L(a):=T^*aT$, $a\in A$, defines a  transfer operator for $\alpha$ and we refer to $aT$, $a\in A$, as \emph{abstract weighted shifts associated with the transfer operators $L:A\to A$}. The endomorphism $\alpha$ of $A\cong C(X)$ induces a continuous map $\varphi:X\to X$. 
Our main aim is to generalize Theorem \ref{thm:lebedev} to this setting.  We achieve it by overcoming the following difficulties:

\textbf{1) Relations.} 
The  $C^*$-algebra $C^*(A,T)$ generated by the operators $aT$, $a\in A$, satisfying (A1), (A2),  can be modeled by the crossed product introduced by Exel 
\cite{exel3}.   Exel's crossed product depends only on the associated transfer operator $L$, see \cite{kwa_Exel}, and therefore
we denote it by $A\rtimes L$.
In general,   
  the relations defining  $A\rtimes L$ as a universal $C^*$-algebra 
are hard to describe explicitly. 
This makes both the definition and the analysis of the structure of $A\rtimes L$
non-trivial. Indeed, the definition of $A\rtimes L$ was in the air for a long time, starting from Arzumanian and Vershik
\cite{Arzu_Vershik1}, \cite{Arzu_Vershik2},  through Cuntz and Krieger \cite{CK}, 
Renault \cite{Renault}, Deaconu \cite{Deaconu}, Lebedev (unpublished notes) and culminating with  Exel \cite{exel3}, whose definition  still required some modifications, see \cite{er} and \cite{kwa_Exel}.
The most studied and best understood case is when the associated map $\varphi:X\to X$ is
a local homeomorphism and a cocycle $\varrho$ associated to $L$ attains strictly positive values. Then $A\rtimes L$ has a groupoid model 
and we say that $L$ is of \emph{finite type} (cf. Proposition \ref{prop:finite_type_characterisation} below). In the finite type case, a `Cuntz-Krieger type' relation  that defines $A\rtimes L$ was made explicit by Exel and Vershik \cite{exel_vershik}. In general, we surpass this problem by assuming 
the following natural nondegeneracy condition: 
\begin{itemize}
\item[(A3)]  $ATH$ is linearly dense in $H$. 
\end{itemize}
We verify axioms (A1), (A2), (A3) for large classes of natural examples (cf. Propositions \ref{prop:T_unitary}, \ref{prop:axiom_A1_concrete}, \ref{prop:abstract_wso_from_measure_system}). In fact, by using a Cuntz-Pimsner model for $A\rtimes L$  \cite{br}, \cite{kwa_Exel} and 
appealing to \cite{Hirshberg} we show that every faithful transfer operator is included in our setting
(see the last part of Theorem \ref{thm:gauge_uniqueness}). Applying results of \cite{CKO} we find
that 
$C^*(A,T)\cong A\rtimes L$ whenever  (A1), (A2), (A3) hold and  $\varphi:X\to X$ is topologically free (Theorem \ref{thm:isomorphism}). 
In the finite type case, using a groupoid model for $A\rtimes L$, we find  that $A\subseteq C^*(A,T)$ is a Cartan inclusion if and only if  $\varphi:X\to X$ is topologically free (Theorem \ref{thm:Cartan_crossed_products}), which slightly extends the main result of \cite{CS}.
Using a version of Ruelle's Perron-Frobenius Theorem (Theorem \ref{thm:Ruelle-Perron-Frobenious}), we show that  every finite type transfer operator $L:C(X)\to C(X)$ for  expanding  $\varphi$ with no non-wandering points has a natural representation on $L^2_{\mu}(X)$ via the composition operator $T$ satisfying (A1), (A2), (A3) (Proposition \ref{prop:cocycle_inducing_measures}). This in particular gives a natural representation of Cuntz-Krieger algebras \cite{CK} (cf. Remark \ref{rem:Cuntz-Krieger_representaion}).

\textbf{2) Spectral radius.} 
A somewhat separate problem is how to describe  spectral radii of operators $aT$, $a \in A$, associated with a  transfer operator $L$. 
It reduces to analysis of the spectral radii of  (weighted unital) transfer operators on $C(X)$, as we have $r(aT)=\sqrt{r(L|a|^2)}$. 
In the full generality we consider, these spectral radii were described by
Antonevich, Bakhtin, Lebedev \cite{t-entropy}, \cite{t-entropy2} via variational principles exploiting a new invariant they introduce - the \emph{$t$-entropy}. We push this research further when $\varphi:X\to X$ is an expanding open map. 
Then the problem reduces to investigation of 
$r(\mathcal{L}_{c})$  for a Ruelle-Perron-Frobenious operator $\mathcal{L}_cf(y)=\sum_{x\in \varphi^{-1}(y)} c(x)f(x)$ 
with an arbitrary continuous `potential' $c:X\to [0,\infty)$.  One of the fundamental principles of
thermodynamical formalism is that   $\ln r(\mathcal{L}_{c})$ 
coincides with \emph{topological pressure} $P(\ln c,\varphi)$.
A number of results in this direction are known, cf.    \cite{Bowen}, \cite{Ruelle0}, \cite{Walt1}, \cite{LatStep}, \cite{Ruelle}, \cite{Lebedev_Maslak}, \cite{Fan_Jiang}, \cite{Przytycki}. 
 However, none of  these sources considers a general case. Usually it is assumed that $\varphi$ is topologically mixing,  $c$ is  H\"older continuous and strictly positive, and the space $X$
is a finite dimensional manifold or a shift space. 
Here we prove that $\ln r(\mathcal{L}_{c})=P(\ln c,\varphi)$ for an arbitrary open expanding map $\varphi:X\to X$  on a compact metric space and an arbitrary continuous $c:X\to [0,\infty)$
(Theorem \ref{thm:Ruelle's}). This result is of independent interest. It 
reveals the relationship between $t$-entropy and Kolmogorov-Sinai entropy (Corollary \ref{cor:t-entropy vs Kolmogorov-Sinai entropy}).
It also gives efficient formulas for the spectral radius $r(aT)$ of operators associated to  expanding open maps  (Corollary \ref{cor:abstract_spectral_radius_expanding}).

\textbf{3) Riesz Projectors.} 
When $T$ is a unitary satisfying \eqref{eq:automorphism_relations}, then every Riesz projector for a part of the spectrum 
$\sigma(aT)$, $a\in A $ belongs to $A\cong C(X)$ and corresponds to a clopen set $X_0\subseteq X$, which is \emph{$\varphi$-invariant} in the sense 
that $\varphi^{-1}({X_0})={X_0}$, see \cite{Anton_Lebed}. 
This is a crucial step in the proof of Theorem \ref{thm:lebedev}. 
In our irreversible  setting this fact  fails  and the analysis of Riesz projectors is much more subtle. 
It may happen that a  Riesz projector $P$  for an abstract weighted shift $aT$, $a\in A$, associated to a transfer operator $L$, does not belong to $A$ and even if it does it may not correspond to a $\varphi$-invariant set (see Example \ref{ex:contrexample}). 
However, developing  new arguments we proved that if $a$ is (close to being) invertible and $\varphi$ is topologically free, then
$P\in A$ and it corresponds to a $\varphi$-invariant set (Proposition \ref{prop:Riesz_projector}). In general,  $P\in A'$ and hence it belongs to $A$ whenever $A\subseteq C^*(A,T)$ is a Cartan inclusion. Then either the corresponding set or its complement is forward $\varphi$-invariant (Proposition \ref{prop:Cartan_Riesz_Projections})

\textbf{Main results.} For reasons explained above we prove two versions of Theorem \ref{thm:lebedev} for operators $aT$, $a\in A$, associated with a transfer operator. 
We get  the full description of $\sigma(aT)$ in the case when  $a$ is (close to being) invertible (Theorem \ref{thm:main_result1}), 
and  for arbitrary $a\in A$ in the finite type case (Theorem \ref{thm:main_result2}). We particularly like the results for expanding maps 
(cf. Corollary \ref{cor:spectrum_expansive}).
In general,  if $\varphi:X\to X$ is topologically free, 
then the spectrum of an operator $aT$, $a\in A$, is rotation invariant and if $A\cap \K(H)=\{0\}$ then  $\sigma(aT)=\sigma_{ess}(aT)$  (Corollary \ref{cor:Spectral_consequences}).
We illustrate our results on examples coming from Cuntz-Krieger algebras (Example \ref{ex:Cuntz-Krieger algebras}), weighted composition operators on $L^2$-spaces (Corollaries \ref{cor:measure_dyna_spectrum}, \ref{cor:abstract_wso_from_measure_system1})
and weighted shifts on directed trees (Example \ref{ex:directed_trees}). 

The paper is organized as follows. In Section \ref{sec:preliminaries} we discuss some basic facts concerning transfer operators 
and the thermodynamical formalism. Section \ref{sec:abstract_WSO} introduces abstract weighted shifts and discusses model examples. 
Variational formulas for spectral radii are studied in Section \ref{sec:Spectral radius}.  
The structure of the associated $C^*$-algebras and spectral consequences of topological freeness are presented in Section \ref{sec:C*-algebras}. 
Conditions under which Riesz projectors for $aT$, $a\in A$,  belong to $A$ are given in Section \ref{sec:Riesz_projectors}. 
Finally, in Section \ref{sec:spectrum} we present the results that describe the whole spectrum of $aT$ and give estimates to 
a number of its connected components.
\section{Preliminaries}\label{sec:preliminaries}
\subsection{Thermodynamical formalism for expanding open maps}
Throughout this paper $\varphi:X\to X$ is a continuous map on a compact Hausdorff space $X$. 
Recall that $X$ is metrizable if and only if $X$ is second countable if and only if 
the $C^*$-algebra $C(X)$ of continuous complex valued functions is separable. 
If $d$ is a metric on $X$, then for any  continuous real valued function $b\in C(X,\R)$ the \emph{topological pressure} of $(X,\varphi)$ with potential $b$ can be defined  by  the formula
\begin{equation}\label{eq:topological_pressure}
P(\varphi, b) :=  \lim_{\varepsilon \to 0} \limsup_{n\to\infty}\sup_{E\subseteq X \text{ is }
\atop \varepsilon\text{-separated  in }d_n} \frac{1}{n} \ln\sum_{y\in E}\exp\left({\sum_{i=0}^{n-1} b(\varphi^i (y))}\right),
\end{equation}
where $d_n(x,y)=\max\limits_{i=0,...,n}d(\varphi^{i}(x),\varphi^{i}(y))$ is  $n$-Bowen's metric.
It is crucial that $P(\varphi, b)$ can be also expressed in terms of the following variational principle:
$$
P(\varphi, b)=\sup_{\mu\in \Inv(X,\varphi)} \left(\int_{{X}}b\,d\mu +h_{\varphi}(\mu)\right),
$$
where  $\Inv(X,\varphi)$ is the set of $\varphi$-invariant Borel  probability measures on $X$ and 
$h_{\varphi}(\mu)$ is the Kolmogorov-Sinai entropy, cf. \cite{Walt2}, \cite{Przytycki}.
A measure $\mu\in \Inv(X,\varphi)$ is called an \emph{equilibrium} state if 
$P(\varphi, b)=\int_{{X}}b\,d\mu +h_{\varphi}(\mu)$. If the entropy map $\Inv(X,\varphi) \ni \mu \mapsto h_{\varphi}(\mu) $ is upper-semicontinuous,
then $P(\varphi, b)=\max_{\mu\in \Inv(X,\varphi)} \left(\int_{{X}}b\,d\mu +h_{\varphi}(\mu)\right)$. 
If  the set of equilibrium measures is non-empty it contains an ergodic measure (whenever $P(\varphi, b)< \infty$) and for a dense subset of functions in $b\in C(X, \R)$ the equilibrium measure is unique, see \cite[9.13 and 9.15.1]{Walt2}. 
The study of unique equilibrium measures was initiated by the theory 
of thermodynamical formalism that works for expanding maps:
\begin{defn}
A continuous map $\varphi:X\to X$ is   \emph{expanding} (expands small distances or is locally expanding) if $X$ is compact and
 can be equipped with  a metric $d$ for which
 there are $\varepsilon >0$  and $\theta >1$ such that $d(\varphi(x),\varphi(y))\geq \theta d(x,y)$ whenever  $d(x, y)<\varepsilon$.
\end{defn}
Being expanding does not depend on the choice of a metric (compatible with the topology). By \cite{Reddy},
$\varphi:X\to X$  is expanding if and only if $\varphi:X\to X$ is \emph{positively expansive}, i.e. for some metric $d$
 there is $\delta >0$  such that for any two points $x\neq y$  we have $d(\varphi^n(x),\varphi^n(y))> \delta $ for some $n\in \N$.
By Schwartzman's theorem, see  \cite{MS}, \cite{KR}, 
a homeomorphism   $\varphi:X\to X$ is expanding (positively expansive) if and only if $X$ is finite. Thus being expanding is a notion designed for 
irreversible dynamics. 
The literature is abundant in the results concerning expansive homeomorphisms, but it is well known that (one-sided) versions of such facts hold for expanding (positively  expansive) maps, cf. \cite[Remark on page 145]{Walt2}, \cite[page 58]{KR}. 
A counterpart of  \cite[Theorem 5.24]{Walt2}, \cite[Corollary 3.3]{KR}   states that any expansive map $\varphi:X\to X$ is a factor of a one-sided subshift:  there
is a closed  subset $Z$ of $\Sigma_{n}=\{1,...,n\}^\N$ with $\sigma(Z)=Z$, where $\sigma:\Sigma_{n}\to \Sigma_{n}$ is one-sided shift, and a continuous surjective map 
$\pi:Z\to X$ such that $\pi \circ \sigma=\sigma\circ \pi$. Moreover, the one-sided subshifts $(Z,\sigma)$, as above, can be characterised as 
 expanding maps on zero-dimensional spaces. 
By \cite[Theorem 3.5.6]{Przytycki}, for expanding maps  the entropy map $\Inv(X,\varphi) \ni \mu \mapsto h_{\varphi}(\mu) $ is upper-semicontinuous
and therefore equilibrium (ergodic) measures for expanding maps always exist.

An expanding map is necessarily locally injective and hence if it is open, it is a \emph{local homeomorphism}.
An open expanding map $\varphi:X\to X$ is \emph{topologically mixing} if and only if for every non-empty open set $U$ there is $n\geq 1$ such that $\varphi^n(U)=X$,
 cf. \cite{Przytycki}. Usually Gibbs measures are considered for topologically mixing maps, cf. \cite{Bowen}, \cite{Ruelle0},  \cite{Walt1}, \cite{Fan_Jiang},
However, as shown in \cite{Przytycki} the theory works well under a slightly weaker assumption that 
$\varphi:X\to X$ is \emph{topologically transitive}, i.e. for any  non-empty open sets $U,V\subseteq X$ there is $n\in \N$ such that
$\varphi^n(U)\cap V\neq \emptyset$. If $\varphi$ is open expanding, this is equivalent to 
existence of $N\in \N$ such that $\varphi^N$ is topologically mixing.
\begin{defn}\label{defn:Gibbs_distribution}
Suppose that $(X,d)$ is a metric compact space. We say that $\mu\in \Inv(X,\varphi)$ is a \emph{Gibbs distribution} for a  continuous map $\varphi:X\to X$  with  potential $b\in C(X,\R)$ 
if for sufficiently small $\varepsilon>0$ there is $P\in \R$,  and $C\geq 1$ such that for all $x\in X$ and $n\geq 1$ we have
$$
C^{-1} \leq \frac{\mu(B_n(x,\varepsilon))}{\exp({\sum_{i=0}^{n-1} b(\varphi^{i}(x)) - nP}) } \leq C
$$
where $B_n(x,r)$ is the open ball in $n$-Bowen metric $d_n(x,y)=\max\limits_{i=0,...,n}d(\varphi^{i}(x),\varphi^{i}(y))$.
\end{defn} 
 \begin{thm}[Existence and uniqueness of Gibbs measure]\label{thm:existence_of_Gibbs_measures}
Let $\varphi:X\to X$ be an open expanding map and let  $b\in C(X,\R)$  be H\"older continuous.  There is a unique Gibbs distribution 
$\mu_{\varphi,b}\in \Inv(X,\varphi)$ for $\varphi$ and $b$ and the constant $P$ in 
Definition \ref{defn:Gibbs_distribution} is the topological pressure  $P(\varphi, b)$. Moreover, $\mu_{\varphi,b}$ is the unique equilibrium state for $\varphi$ and $b$, and 
for every $y\in X$ we have
$$
P=P(\varphi, b)=\int_{{X}}b\,d\mu_{\varphi,b} +h(\mu_{\varphi,b})=\lim_{n\to\infty}  \frac{1}{n} \ln\sum_{x\in \varphi^{-1}(y)}\exp\left({\sum_{i=0}^{n-1} b(\varphi^i (x))}\right).
$$
\end{thm}
\begin{proof}
Combine \cite[Propositions 3.4.3, 4.4.1,  4.1.5 and Theorems  4.3.2,   4.6.2]{Przytycki}.
\end{proof}
 Gibbs distributions $\mu_{\varphi,b}$ are usually constructed by means of  Ruelle's Perron-Frobenius Theorem. More specifically, if  $\varphi:X\to X$ is a local homeomorphism, then for every continuous
 $c:X\to [0,+\infty)$ the formula 
\begin{equation}\label{eq:Ruelle-Perron_Frobenious}
\mathcal{L}_cf(y)=\sum_{x\in \varphi^{-1}(y)} c(x)f(x)
\end{equation}
defines a positive linear operator  $\mathcal{L}_c:C(X)\to C(X)$, which is called the \emph{Ruelle-Perron-Frobenius transfer operator} associated to $\varphi$ and $c$.
By positivity the dual operator $\mathcal{L}_c^*:C(X)^*\to C(X)^*$ restricts to an affine map on the cone of  finite Borel measures. 
The following theorem was originally proved  for one-sided shifts by Ruelle. Then it was generalized by a number authors to  mixing expanding maps \cite{Bowen}, \cite{Walt1}, \cite{Ruelle0}, \cite{Fan_Jiang}.  We generalize it a bit further, to topologically transitive expanding maps:
\begin{thm}[Ruelle's Perron-Frobenius Theorem]\label{thm:Ruelle-Perron-Frobenious}
Let $\varphi:X\to X$ be an open expanding and topologically transitive map on a compact metric space $X$. 
Suppose that $b\in C(X,\R)$ is H{\"o}lder continuous. Put $c:=\exp(b)$ and let $\rho:=r(\LL_{c})$ be the spectral radius of the associated
 Ruelle-Perron-Frobenius transfer operator $\mathcal{L}_{c}:C(X)\to C(X)$. 
\begin{enumerate}
\item \label{enu:Ruelle-Perron-Frobenious1}There is a unique  Borel probability measure $\nu$ on $X$  such that $\LL_{c}^* \nu =\rho \nu$.
\item \label{enu:Ruelle-Perron-Frobenious2} $\rho$ is a strictly positive  eigenvalue of  $\LL_{c}$   with a unique strictly positive eigenfunction $h\in C(X)^+$
 satisfying $\nu(h)=1$. 


\end{enumerate}
With $h$ and $\nu$ as above, $h\nu(a):=\nu(ah)$, $a\in C(X)$, is the Gibbs measure $\mu_{\varphi,b}$ for $\varphi$ and $b$. Moreover, we 
have $P(\varphi,b)=\ln r(\LL_{c})$.
\end{thm}
\begin{proof} For brevity put $\LL:=\LL_{c}$ and let $\M(X)$ be the set of Borel probability measures.

\ref{enu:Ruelle-Perron-Frobenious1}. 
Existence of $\nu\in \M(X)$  with $\LL^*(\nu)=c\nu$ for some $c>0$ follows from the  Schauder-Tychonoff fixed point theorem, see \cite[Theorem 5.2.8]{Przytycki}.  Suppose that $\mu,\eta\in \M(X)$ are such that $\LL^*(\mu)=c_1\mu$  and
 $\LL^*(\eta)=c_2\eta$ for some $c_1,c_2>0$. By \cite[Propositions 5.2.11, 5.1.14]{Przytycki},  
$c_1=c_2=e^{P(\varphi,b)}$ and the measures $\mu$
and $\eta$ are equivalent to the Gibbs measure $\mu_{\varphi,b}$. We also have $P(\varphi,b)=\ln r(\LL)=\ln \rho$, which under present assumptions could be readily proved using the formula in  Theorem \ref{thm:existence_of_Gibbs_measures}.  We will show it without any assumption on $b$ in Theorem \ref{thm:Ruelle's} below.
 Hence $\LL^*(\mu)=\rho\mu$,
 $\LL^*(\eta)=\rho\eta$  and  there is an $\eta$-integrable function $h$ such that $\mu=h \eta$.
 For every $a\in C(X)$ we have
$\LL(a)h=\LL(a (h\circ \varphi))$ and therefore  $\LL^*(h \eta)=(h\circ \varphi)\LL^*(\eta)=\rho(h\circ \varphi)\eta$. 
Using this we get
$$
h\eta=\mu=\rho^{-1}\LL^*(\mu)=\rho^{-1}\LL^*(h \eta)=(h\circ \varphi)\eta.
$$
This implies that $h=(h\circ \varphi)$ $\eta$-almost everywhere. But since $\eta$ is equivalent to 
$\mu_{\varphi,b}$, which is $\varphi$-ergodic, cf. \cite[Corollary 5.2.13]{Przytycki},
it follows that $h$ is constant $\eta$-almost everywhere. Thus $\eta=\mu$. 

\ref{enu:Ruelle-Perron-Frobenious2}.
Existence of $h\in C(X)^+$ with $\nu(h)=1$ and $\LL(h)=\rho h$ follows from \cite[Proposition 5.3.1]{Przytycki}.
Then $h\nu=\mu_{\varphi,b}$ is the Gibbs measure by \cite[Theorem 5.3.2]{Przytycki}. Thus if $g\in C(X)^+$ is 
 another function with properties of $h$, then  $h\nu=g\nu$, that is $g=h$, $\nu$-almost everywhere, which implies that $g=h$ because   $g$ and $h$ are continuous and $\supp\nu=X$.
 \end{proof}
We explain how the above results can be applied in  more general situations  using a  spectral decomposition.
We say that a set $U\subseteq X$ is \emph{forward $\varphi$-invariant} if $\varphi(U)\subseteq U$ and it is  \emph{$\varphi$-invariant} if 
$\varphi^{-1}(U)=U$.
A point $x\in X$ is \emph{non-wandering} if for every open
neighborhood $V$  of $x$  we have $V\cap \varphi^n(V)\neq \emptyset$ for
some $n\in \N$.  The set of non-wandering points is denoted by $\Omega(\varphi)$. It is  closed, forward $\varphi$-invariant and $\supp \mu\subseteq \Omega(\varphi)$ for every $\mu\in\Inv(X,\varphi)$.
If $\varphi$ is open expanding, then $\Omega(\varphi)$ is equal to the closure of the set of all periodic points
 \cite[Proposition 4.3.6]{Przytycki}.
By  \cite[Corollary 4.2.4]{Przytycki}, the following is a special case of \cite[Theorem 4.3.8]{Przytycki}:
\begin{thm}[Spectral decomposition]\label{thm:spectral decomposition} 
Suppose that $\varphi:X\to X$ is an open expanding map with $\Omega(\varphi)=X$.
Then $X=\bigsqcup_{j=1}^N X_j$ is a union of finitely many $\varphi$-invariant disjoint clopen  sets $X_j$, $j=1,..., N$, such that
for each $j$, the map $\varphi|_{X_j}:X_j\to X_j$ is topologically transitive. Moreover, each 
 $X_j$ is the union of $k(j)$ disjoint clopen sets $X_j^k$  which are cyclically
permuted by $\varphi$ and such that $\varphi^{k(j)}|_{X_j^k}:X_j^k\to X_j^k$ is topologically mixing, for all $j$ and $k$.
\end{thm}
\begin{ex}\label{ex:TMS} Let $\mathbb{A}=[A(i,j)]_{i,j=1}^n$ be a matrix of zeros and ones, with no zero rows.
The \emph{topological Markov shift} 
with  states $\{1,..,n\}$ and the transition matrix $\mathbb{A}$ is the map 
$\sigma_{\mathbb{A}}:\Sigma_{\mathbb{A}}\to \Sigma_{\mathbb{A}}$ where
$$
\Sigma_{\mathbb{A}}:=\{(\xi_1,\xi_2,...)\in \{1,...,n\}^\N: A(\xi_i,\xi_{i+1})=1\}, \quad \sigma_{\mathbb{A}} (\xi_1,\xi_2,\xi_3,...):=(\xi_2,\xi_3,...).
$$
Here $\Sigma_{\mathbb{A}}$ is a metrizable compact space equipped with 
 the product topology inherited from $\{1,...,n\}^\N$, and $\sigma_{\mathbb{A}}$
is an expanding local homeomorphism.  The above data can also be viewed as a directed
graph with  vertices $\{1,...,n\}$ and  edges $\{(i,j): A(i,j) = 1\}$. Then $\Sigma_{\mathbb{A}}$
is the space of infinite paths in this graph. The number of paths from $i$ to $j$ of length $n$ is $A(i,j)^{(n)}$ where  $\mathbb{A}^n=[A(i,j)^{(n)}]_{i,j=1}^n$ is the $n$-th power of 
the matrix 
$\mathbb{A}$. We write $i\to j$ if 
there is a path from $i$ to $j$, that is $A(i,j)^{(n)}>0$ for some $n$.
A state $i$ is called \emph{essential}  if $i\to j$ implies $j\to i$ for every  state $j$. We  have 
$
\Omega(\sigma_{\mathbb{A}})=\{(\xi_1,\xi_2,...)\in \Sigma_{\mathbb{A}}: \xi_i\text{ is essential for every }i=1,2,...\}.
$
Hence $\Omega(\sigma_{\mathbb{A}})=\Sigma_{\mathbb{A}}$ if and only if all states  are essential.
Also  $\sigma_{\mathbb{A}}$ is topologically transitive if and only if $\mathbb{A}$ is \emph{irreducible} (i.e. for every $i$ and $j$ there is $n$ 
such that $A(i,j)^{(n)}>0$) if and only if $i\leftrightarrow j$ for all states $i$ and $j$.
The map  $\sigma_{\mathbb{A}}$ is topologically mixing if and only if $\mathbb{A}$ is \emph{aperiodic}, that is there is $n\in \N$ such that all entries in  $\mathbb{A}^n$
are non-zero.   
If $\Omega(\sigma_{\mathbb{A}})=\Sigma_{\mathbb{A}}$, then $i\leftrightarrow j$ is an equivalence relation 
and  thus   $\{1,...,n\}$ decomposes into disjoint equivalence classes.
 After possible relabeling, these classes are  $\{k_{j}+1,...,k_{j+1}\}$, $j=1,...,N$, where  $k_1=0<k_2<...<k_{N+1}=n$. Then $\mathbb{A}=\text{diag}(\mathbb{A}_1,...,\mathbb{A}_N)$
where $\mathbb{A}_j=[A(i,l)]_{i,l=k_{j}+1}^{k_{j+1}}$ for $j=1,...,n$ and 
$
\Sigma_{\mathbb{A}}=\bigsqcup_{j=1}^N \Sigma_{\mathbb{A}_j}
$
is the decomposition from Theorem \ref{thm:spectral decomposition} applied to $\sigma_{\mathbb{A}}$.
\end{ex}

\subsection{Endomorphisms, transfer operators and local homeomorphisms}
In this paper $A$ will always stand for a\emph{ unital commutative $C^*$-algebra}. So 
 by the Gelfand-Naimark theorem we may identify $A\cong C(X)$ with the algebra of continuous functions on a compact Hausdorff space $X$. By a unital \emph{endomorphism} of $A$ we mean a $*$-preserving endomorphism $\alpha:A\to A$ such that $\alpha(1)=1$. 
 Unital endomorphisms $\alpha:C(X)\to C(X)$  are in the one-to-one correspondence with continuous mappings $\varphi:X\to X$ given by
$$
\alpha(a)(x)=a(\varphi(x)), \qquad a\in C(X),\,\, x\in X,
$$
see, for instance, \cite[Proposition 1.1]{kwa-leb}. If $\alpha$ and $\varphi$ are in this correspondence, then $\alpha$ is a monomorphism ($\alpha$ is injective) if and only if $\varphi$ is surjective.
 Transfer operators on  $C^*$-algebras were introduced in \cite{exel3}.
In this paper we will consider  transfer operators that preserve the unit.
\begin{defn} A  \emph{(unital) transfer operator}  on a $C^*$-algebra $A$ is  a unit preserving  positive linear map $L:A\to A$ for which there is an endomorphism 
$\alpha:A\to A$ such that
\begin{equation}\label{eq:transfer_operator_relation}
L(\alpha(a)b)=a L(b), \qquad a\in A. 
\end{equation}
Then we also say that $L$ is a transfer operator for $\alpha$. 
\end{defn}
\begin{rem}
Since we assume that $L(1)=1$, relation \eqref{eq:transfer_operator_relation} implies that $L$ is a left inverse to $\alpha$. 
Hence if an endomorphism  admits a transfer operator, then it is necessarily a monomorphism and therefore the dual map 
$\varphi:X\to X$ is surjective.
\end{rem}
\begin{rem}\label{rem:on_faithful_transfers}
As a rule there are many different transfer operators 
for a given endomorphism $\alpha:A\to A$. However, if a transfer operator $L:A\to A$ is \emph{faithful}, 
i.e. it is injective on the  positive elements, 
then there is exactly one endomorphism $\alpha$ satisfying \eqref{eq:transfer_operator_relation}, see \cite[Proposition 4.18]{kwa_Exel}. 
\end{rem}
Transfer operators  for a  monomorphism $\alpha:C(X)\to C(X)$, with the dual map $\varphi:X\to X$,
are in a one-to-one correspondence with  
continuous maps $X\ni x \longmapsto \mu_x \in \M(X)$, where $\supp\mu_x\subseteq \varphi^{-1}(x)$ for every $x\in X$,
and $\M(X)$ is the space of Borel probability measures on $X$ equipped with the weak$^*$ topology.
Every transfer operator for $\alpha$ is of the form 
\begin{equation}\label{equ:transfer_operator_form}
L(a)(y)=\int_{\varphi^{-1}(x)} a(y) d\mu_x(y), \qquad a\in C(X).
\end{equation}
 When $\varphi:X\to X$ is a local homeomorphism, then every transfer operator for $\alpha$ is of the form
\begin{equation}\label{equ:transfer_operator_form2}
L(a)(y)=\sum_{x\in\varphi^{-1}(y)}\varrho(x)a(x), \qquad a\in C(X)
\end{equation}
where $\varrho:X\to [0,1]$  is a continuous function such that 
$\sum_{x\in\varphi^{-1}(y)}\varrho(x)=1$ for all $y\in X$. 
We call a map $\varrho$ with the above properties  a \emph{cocycle} for the map $\varphi:X\to X$. 
 For details concerning the above discussion see, for instance, \cite{kwa3}.
\begin{ex}\label{ex:Ruelle's_operators}
 Ruelle's transfer operator $\mathcal{L}_{c}f(y)=\sum_{x\in \varphi^{-1}(y)}c(x)f(x)
$ associated to a local homeomorphism $\varphi:X\to X$ and $c \geq 0$  in general fails to be unital.
Nevertheless, if $c>0$, then $\mathcal{L}_{c}(1)>0$ is invertible and therefore 
$L:=(\mathcal{L}_{c}(1))^{-1}\mathcal{L}_{c}$ is a (unital) transfer operator with the cocycle 
$
\varrho(x)=c(x)(\sum_{x'\in \varphi^{-1}(\varphi(x))}c(x'))^{-1}  >0
$.
\end{ex}
We end this section by `an algebraic characterization' of transfer operators  from  Example \ref{ex:Ruelle's_operators}.
 Let $B\subseteq A$ be a unital $C^*$-subalgebra of the $C^*$-algebra $A$, so that $B$ contains the unit of $A$. 
A \emph{conditional expectation} from $A$ onto $B$ is a linear contractive projection $E:A\to B\subseteq A$ onto $B$.
  By  Tomiyama's theorem  this is equivalent to saying that $E:A\to B\subseteq A$ is a unit preserving completely positive linear map such that 
	$
E(ba)=b E(a)$ for all $a\in A$ and $b\in B$. We have the following easy fact, see \cite{exel3}, \cite{kwa3}.
\begin{lem}\label{lem:expectations_vs_transfer_operators}
Let $\alpha:A\to A$ be a  monomorphism. We have a one-to-one correspondence between conditional expectations 
$E:A\to \alpha(A)\subseteq A$  and transfer operators $L:A\to A$ for $\alpha$. It is given by 
relations
$
E=\alpha\circ L$ and  $L(a)= \alpha^{-1}(E(a))$ for $ a\in A$.
\end{lem}
A conditional expectation $E:A\to B\subseteq A$ onto $B$ is said to be of \emph{finite type} if there is a quasi-basis for $E$, i.e. there exists  a finite set $\{u_1, . . . , u_N\}\subseteq A$ such that
$
a =\sum_{i=1}^N  u_i E(u_i^*a)$,  for all $a\in A$,
see \cite{Watatani}.
\begin{defn}[cf. \cite{exel4} Definition 8.1] Let  $\alpha:A\to A$ be a monomorphism.
 We say   that $\alpha$ is of \emph{finite-type} if there is a conditional expectation of finite type onto the range $\alpha(A)$ of $\alpha$. 
Equivalently, in view of Lemma \ref{lem:expectations_vs_transfer_operators}, $\alpha$ is of finite type if and only if there is a transfer operator $L:A\to A$ and elements
$\{u_1, . . . , u_N\}\subseteq A$  such that
$$
a =\sum_{i=1}^N  u_i \alpha\big( L(u_i^*a)\big),\quad \textrm{ for all }a\in A.
$$   
In this case we will also say that $L$ is a  \emph{transfer operator  is of finite type}.
\end{defn}

\begin{prop}\label{prop:finite_type_characterisation}
Let $\alpha:C(X)\to C(X)$ be a unital monomorphism and let $\varphi:X\to X$ the associated continuous surjective map. Then
$
\alpha$  is of finite-type if and only if  $\varphi$ is a local homeomorphism.
Moreover, a transfer operator $L:C(X)\to C(X)$  for $\alpha$ is of finite type if and only if  it is given by \eqref{equ:transfer_operator_form2} 
where  the cocycle $\varrho:X\to [0,1]$ attains strictly positive values. Then the associated quasi-basis can be defined by
\begin{equation}\label{equ:quasi_basis_form}
u_i(x)=\sqrt{\frac{v_i(x)}{\varrho(x)}}, \qquad i=1,...,n
\end{equation}
where   $\{v_i\}_{i=1}^n\subseteq C(X)$ is a  partition of unity  subordinated to an open cover $\{U_i\}_{i=1}^n$ of $X$, with $\varphi:U_i\to \varphi(U_i)$ is a homeomorphism for each 
$i=1,...,n$. 
\end{prop}
\begin{proof}
Assume $\varphi$ is a local homeomorphism. Putting $\varrho(x):=|\varphi^{-1}(x)|^{-1}>0$ we get a continuous cocycle for $\varphi$.
 Moreover, for any continuous cocycle $\varrho:X\to (0,1]$, the transfer operator given by \eqref{equ:transfer_operator_form2} 
is of finite type. Indeed,  for any $a\in C(X)$ and $\{u_i\}_{i=1}^n$ given by \eqref{equ:quasi_basis_form} we have
\begin{align*}
\Big[\sum_{i=1}^N  u_i \alpha\big( L(u_i^*a)\big)\Big] (x)&=\sum_{i=1}^N  u_i(x)  \sum_{\varphi(y)=\varphi(x)}\overline{u_i(y)} \varrho(y)a(y)
\\
&=\sum_{i=1}^N  |u_i(x)|^2 \varrho(x)a(x)=\sum_{i=1}^N  v_i(x) a(x)=a(x).
\end{align*}
 Hence $\alpha$ is of finite type, cf. \cite[Proposition 8.2]{exel_vershik}.

Now let $L$ be any transfer operator  for $\alpha$ which is  of finite type.   Fix a corresponding quasi-basis $\{u_n\}_{n=1}^N\subseteq A$. Then  $L$ is given by
 \eqref{equ:transfer_operator_form} and we have 
\begin{equation}\label{equation for contradiction with Kronecker Capelli}
a(x)=\sum_{i=1}^{N}  u_i(x) \int_{\varphi^{-1}(\varphi(x))} \overline{u_i(y)}a(y) d\mu_{\varphi(x)} (y) \quad\textrm{for every } a\in C(X) \textrm{ and }x\in X.
\end{equation}
This forces  $\varphi$ to be  at most  $N$-to-one map. Indeed, suppose on the contrary that there are $N+1$ different points $x_1,...,x_{N+1}\subseteq \varphi^{-1}(\varphi(x))$ for a certain $x\in X$. Then putting $c_{j,i}:=u_i(x_j)$ we get a  matrix $C=[c_{j,i}]_{j=1,i=1}^{N+1,N}$. Its range is at most $N$-dimensional. Hence there is a vector 
$y=[y_j]_{j=1}^{N+1}$ such that for every vector $\lambda=[\lambda_i]_{i=1}^{N}$ we have $C \lambda \neq y$. This contradicts \eqref{equation for contradiction with Kronecker Capelli} by taking $a \in C(X)$ such that $a(x_j)=y_j$ for all $j=1,...,N+1$, and putting $\lambda_i:=\int_{\varphi^{-1}(\varphi(x))} \overline{u_i(y)}a(y) d\mu_{\varphi(x)} (y)$, for $i=1,...,N$.

  We  define $\varrho(x):=\mu_{\varphi(x)}(\{x\})$, $x\in X$.
Since $\varphi$ is at most $N$-to-one,  \eqref{equation for contradiction with Kronecker Capelli} gives
$$
a(x)=\sum_{\varphi(x)=\varphi(y)}a(y)\varrho(y)\sum_{i=1}^{N}  u_i(x)\overline{u_i(y)} \quad\textrm{for every } a\in C(X) \textrm{ and }x\in X.
$$
Plugging into this equation   $a\in C(X)$ such that $a(x)=1$ and $a(\varphi^{-1}(\varphi(x))\setminus\{x\} )=0$  we get 
$1=\varrho(x)\sum_{i=1}^N  |u_i(x)|^2$, for every  $x\in X$. 
This implies that $\varrho$ is strictly positive and continuous. 
In particular,   $\supp\mu_y=\varphi^{-1}(y)$ for every $y\in X$, and therefore $\varphi$ is an open map, see, for instance,  \cite[Lemma 3.28]{kwa_Exel}. 
Thus it suffices to show that every point $x\in X$ has an open neighborhood on which $\varphi$ is injective. 
Note that $X\in y\mapsto L(1/\varrho)(y)= |\varphi^{-1}(y)|$ is   continuous. Hence there is a neighborhood $V$ of 
$\varphi(x)$ such that each point in $V$ has exactly $M\leq N $ elements in the  pre-image. 
Let $\varphi^{-1}(\varphi(x))=\{ x_1, ...,x_M\}$ and take disjoint open sets $\{U_{i}\}_{i=1}^N$ 
such that $x_i \in U_i\subseteq \varphi^{-1}(V)$, $i=1,...,M$. Putting $V':=\bigcap_{i=1}^M \varphi(U_i)$ 
and $U_i':=U_i\cap \varphi^{-1}(V')$, for $i=1,..., M$, we find that $\varphi$ restricted to each $U_i'$ is injective.
Indeed, for every $x' \in U_i'$, the $M$-element set $\varphi^{-1}(\varphi(x'))$ intersects each of  disjoint sets $U_j'$, $j=1,...,M$.
Accordingly, $x$ belongs to some $U_i'$, and hence $\varphi$ is locally injective.
\end{proof}
\begin{cor}\label{cor:transfers_of_finite_type} If  $L:C(X)\to C(X)$  is a transfer operator of finite type, then it is faithful, there is a unique endomorphism $\alpha:C(X)\to C(X)$ such that
$L$ is a transfer operator for $\alpha$ and the map $\varphi:X\to X$ dual to $\alpha$ is a surjective local homeomorphism.
\end{cor}
\begin{proof}
Combine Proposition \ref{prop:finite_type_characterisation} and Remark \ref{rem:on_faithful_transfers}.
\end{proof}
\section{Abstract weighted shift operators}\label{sec:abstract_WSO}

\subsection{The axioms}
We denote by $\B(H)$ the $C^*$-algebra of all bounded operators on a Hilbert space $H$.
Throughout this subsection we assume that $A\subseteq \B(H)$ is a  unital commutative $C^*$-subalgebra and  $T\in \B(H)$ is an isometry.
We will consider the following conditions that $A$ and $T$ may satisfy:
\begin{itemize}
\item[(A1)] $
T^*AT\subseteq A$;
\item[(A2)] there is a unital endomorphism $\alpha:A\to A$ such that 
$
Ta=\alpha(a)T$,  for all $a\in A$;
\item[(A3)] $\overline{ATH}=H$, i.e. the closed linear span of elements of $ATH$ is equal to $H$.
\end{itemize}
Axioms (A1) and (A2)  appear, for instance, in \cite{Lebedev_Maslak},  \cite{t-entropy},
and are sufficient for the study of  the spectral radius $r(aT)$, $a\in A$.
The third  axiom (A3)  will be crucial in the analysis of the  $C^*$-algebra $C^*(A,T)$ and spectra of
the operators $aT$, $a\in A$. 
\begin{prop}\label{prop:faithfulness_of_transfers}  (A1) and (A2) are satisfied if 
and only if  the formula $L(a):=T^*aT$ defines a transfer operator $L:A\to A$ for a unital endomorphism  of $A$. 
If in addition condition (A3) holds, then $L$ is  faithful and the endomorphism in question  is  unique.
\end{prop}
\begin{proof}

Condition (A1) is equivalent to  that the formula $L(a):=T^*aT$, $a\in A$, gives a well defined map  $L:A\to A$. 
Assume (A1).  Then $L:A\to A$ is  positive (because $L(a^*a)=(aT)^*aT \geq 0$) and unit preserving (because $T$ is an isometry). 
Condition (A2) implies that 
$L(b\alpha(a))=T^*b\alpha(a)T= T^*bTa=L(b)a$, for all $a,b\in A$, and hence $L$ is a  transfer operator.
Conversely, assume that $L$ is a transfer operator for an endomorphism $\alpha:A\to A$. Then using the $C^*$-equality
and the  transfer identity \eqref{eq:transfer_operator_relation}, for every $a\in A$, one gets
$
\|Ta-\alpha(a)T \|^2=\|(Ta-\alpha(a)T)(Ta-\alpha(a)T)^* \|=0,
$
which is (A2), cf. \cite[Proposition 4.3]{kwa_Exel}.

Now assume (A1), (A2), (A3). 
Let $a\in A$ be such that $L(a^*a)=0$.  Then $\|aT\|^2=\|T^*a^*aT\|=\|L(a^*a)\|=0$, that is $aT=0$. Since $A$ is commutative this implies that 
for every $b\in A$ and $h\in H$ we have $a (bTh)=b (aT)h=0$. By (A3) elements $bTh$ are linearly dense in $H$. Hence $a=0$.
This proves faithfulness of $L$. Uniqueness of $\alpha$ follows now from \cite[Proposition 4.17]{kwa_Exel}.
\end{proof}

\begin{defn}\label{AWSO}
Let $A\subseteq \B(H)$ be a unital $C^*$-subalgebra and let $T\in \B(H)$ be an isometry satisfying (A1) and (A2). 
We call operators of the form $aT$, $a\in A$,  \emph{abstract weighted shift operators}
 associated with the transfer operator $L:A\to A$ where $L(a):=T^*aT$, $a\in A$. If in addition (A3) holds, we say that the weighted shift operators 
$aT$, $a\in A$, are \emph{well presented}.
\end{defn}
We will be mostly interested in the case when $T$ is not invertible, as otherwise
we land in a situation described in \cite{Anton_Lebed}:
\begin{prop}\label{prop:T_unitary} Assume  (A1) and (A2) where $A$ is a commutative $C^*$-algebra. The following are equivalent:
\begin{enumerate}
\item \label{enu:invertible0} $T$ is a unitary;
\item \label{enu:invertible1} $aT$ is invertible for some $a\in A$;
\item \label{enu:invertible2} $\alpha:A\to A$  is an automorphism and (A3) holds;
\item \label{enu:invertible3}  $T$ is a unitary and $T^*AT=A$, i.e. $aT$, $a\in A$, are abstract weighted shifts sense of \cite[3.1]{Anton_Lebed}.
\end{enumerate}
Accordingly, if $T$ is not invertible, then all operators $aT$, $a\in A$, are not invertible.
\end{prop}
\begin{proof} Implications \ref{enu:invertible2}$\Leftarrow$\ref{enu:invertible3}$\Rightarrow$\ref{enu:invertible0}$\Rightarrow$\ref{enu:invertible1} are obvious. 
If $aT$ is invertible, for some $a\in A$, then  it follows from Lemma \ref{lem:Riesz_projector}\ref{enu:Riesz_projection2} (which we will prove below) that  $a$ is  invertible.
Hence also $T$ is invertible and therefore  \ref{enu:invertible0} and \ref{enu:invertible1} are equivalent. 

\ref{enu:invertible2}$\Rightarrow$\ref{enu:invertible0}.
If $\alpha$ is an automorphism, then $L$ is an automorphism, as a right inverse to $\alpha$. 
Multiplicativity of $L$ is equivalent to  $TT^*\in A'$, see \cite[Lemma 1.2]{kwa-leb0}. 
Hence  we get
$ATH=ATT^*H=TT^*AH$, which by (A3) implies that $T$ is onto, and therefore a unitary.

\ref{enu:invertible0}$\Rightarrow$\ref{enu:invertible3}. If $T$ is  a unitary, then (A2) implies that $\alpha(a)=TaT^*$, for $a\in A$.
Hence  $\alpha$ and $L$ are mutually inverse authomorphisms of $A$. In particular,  $T^*AT=A$.
\end{proof}

 When $L$ is of finite type, then condition (A3) can be expressed in terms of a single algebraic relation that was first discovered in \cite{exel_vershik}:
 \begin{lem}\label{lem:A3_characterisation_for_finite_type}
Assume (A1), (A2) and suppose that the associated transfer operator $L:A\to A$ is of finite type. The following conditions are equivalent:
\begin{enumerate}
\item\label{enu:A3_characterisation_for_finite_type1} condition (A3) holds;
\item\label{enu:A3_characterisation_for_finite_type2} we have $\sum_{i=1}^N u_iTT^*u_i^*=1$ for every quasi-basis $\lbrace u_i\rbrace_{i=1}^N\subseteq A$ for $L$;
\item\label{enu:A3_characterisation_for_finite_type3} we have $\sum_{i=1}^N u_iTT^*u_i^*=1$ for some quasi-basis $\lbrace u_i\rbrace_{i=1}^N\subseteq A$ for $L$.
\end{enumerate}
\end{lem}
\begin{proof}\ref{enu:A3_characterisation_for_finite_type1}$\Rightarrow$\ref{enu:A3_characterisation_for_finite_type2}.
Let $\{u_i\}_{i=1}^N$ be a quasi-basis for $L$. For any $a\in A$ and $h\in H$  we have 
$$
\left(\sum_{i=1}^N u_iTT^*u_i^* \right)aTh=\sum_{i=1}^Nu_iTL(u_i^*a)h=\sum_{i=1}^Nu_i\alpha(L(u_i^*a))Th=aTh.
$$
Since elements  $aTh$ are linearly dense in $H$, this implies $\sum_{i=1}^N u_iTT^*u_i^*=1$.

Implication \ref{enu:A3_characterisation_for_finite_type2}$\Rightarrow$\ref{enu:A3_characterisation_for_finite_type3} is clear. 
To show \ref{enu:A3_characterisation_for_finite_type3}$\Rightarrow$\ref{enu:A3_characterisation_for_finite_type1},
suppose that $\{u_i\}_{i=1}^N$ is a quasi-basis for $L$ such that $\sum_{i=1}^N u_iTT^*u_i^*=1$. Then for any $h\in H$ we have
$
h=\sum_{i=1}^N u_iTT^*u_i^*h \in \overline{ATH}
$, that is  $\overline{ATH}=H$.
\end{proof}

In this paper we  consider the situation where $A\cong C(X)$ is commutative, $\alpha$ is given by a continuous surjective map $\varphi:X\to X$, 
and $L$ is given by a continuous family of measures $\{\mu_x\}_{x\in X}$ with $\supp \mu_x \subseteq \varphi^{-1}(x)$ for all $x\in X$. Thus we associate to abstract weighted shifts the triple $(X,\varphi, \{\mu_x\}_{x\in X})$.
In finite type case, $\varphi$ is a local homeomorphism and  $\{\mu_x\}_{x\in X}$ is equivalent to a continuous cocycle map $\varrho:X\to [0,1]$. 

\subsection{Weighted shifts  associated with  topological dynamical systems}
Suppose now that we have a Borel probability measure $\mu\in  \M(X)$   on a compact Hausdorff space $X$ and that $\mu$ has 
 full support, that is $\supp\mu= X$.
Then putting $H:=L^2_\mu(X)$ the $C^*$-algebra of continuous functions $C(X)$ is naturally isomorphic to the $C^*$-algebra $A\subseteq \B(H)$ of operators of multiplication:
$
(ah)(x)=a(x)h(x)$,  $h\in L^2_\mu(X)$, $a\in C(X)$.
Therefore in the sequel we will write $A=C(X)$. Let $\varphi:X\to X$ be a  continuous map that preserves $\mu$. 
Then $\varphi$ is necessarily surjective
   and the composition operator  $T:L^2_\mu(X)\to L^2_\mu(X)$:
\begin{equation}\label{eq:concrete_composition_operator}
(Th)(x)=h(\varphi(x)), \qquad  h\in L^2_\mu(X),
\end{equation}
 is an isometry. We have  a unital monomorphism  $\alpha:A\to A$, 
$\alpha(a):=a\circ \varphi$. Clearly, $T$ and $\alpha$ satisfy axiom (A2). 
Note that $TH=\{ f\circ\varphi: f\in L^2_{\mu}(X)\}$ contains 
the constant functions. So $ATH$ contains the subspace $C(X)\subseteq L^2_{\mu}(X)$ which is dense  in $H=L^2_{\mu}(X)$. 
Hence axiom (A3) holds. However, axiom (A1) may fail  (see Example \ref{ex:not_continous_dissintegration} below). More specifically, suppose that  \emph{$\{\mu_x\}_{x\in X}\subseteq  \M(X)$ is a  disintegration of 
$\mu$ along the fibers of $\varphi$}, i.e. each $\mu_x$ is  supported on $\varphi^{-1}(x)$ 
and $\int_X h(y) d\mu(y)=\int_X \Big(\int_{\varphi^{-1}(y)} h(x) d\mu_y(x)\Big) d\mu(y)$  
 for every  $h\in L^1_\mu(X)$.\footnote{It is well known that such an disintegration exists  when \(X\) is second countable. In our set up  its existence follows from axiom (A3) without any separability assumptions, see Proposition \ref{prop:axiom_A1_concrete}} Then 
\begin{equation}\label{eq:concrete_adjoint_operator}
T^*h(y)=\int_{\varphi^{-1}(y)} h(x) d\mu_y(x), \qquad  h\in L^2_\mu(X),
\end{equation}
and   (A1) holds if and only if $\{\mu_x\}_{x\in X}$ is \emph{continuous} in the sense that the map $X\ni x \mapsto \mu_x \in \M(X)$ 
is continuous. We have a number of characterisations of this continuity:
\begin{prop}\label{prop:axiom_A1_concrete}
Let  $\varphi:X\to X$ be a continuous map for which there exists $\mu\in \Inv(X,\varphi)$  with $\supp \mu= X$.
For the $C^*$-algebra $C(X)\cong A\subseteq \B(L^2_\mu(X) ) $ and the isometry $T\in \B(L^2_\mu(X) )$ as above, axioms (A2) and (A3) always hold. Moreover, the following statements are equivalent:
\begin{enumerate}
\item\label{enu:axiom_A1_concrete1} axiom (A1) holds; 
\item\label{enu:axiom_A1_concrete2} $L^*(\mu)=\mu$ for some transfer operator $L:C(X)\to C(X)$ for $\alpha(a)=a\circ \varphi$;
\item\label{enu:axiom_A1_concrete3} there is a continuous $\{\mu_x\}_{x\in X}$  disintegration of  $\mu$
along the fibers of $\varphi$;
\item\label{enu:axiom_A1_concrete4} the co-isometry $T^*$ preserves the subspace $C(X)\subseteq L^2_\mu(X)$;
\item\label{enu:axiom_A1_concrete5} the projection $TT^*$ preserves the subspace $C(X)\subseteq L^2_\mu(X)$.
\end{enumerate}
 If in addition $\varphi$ is a local homeomorphism, 
and $\{U_i\}_{i=1}^n$ is an open cover of $X$ such that $\varphi:U_i\to \varphi(U_i)$ is injective, for $i=1,...,n$,  then the above conditions are equivalent to
\begin{enumerate}  
\item[(vi)] there is a continuous function $\varrho:X\to [0,1]$ such that 
$\varrho(x)=  \frac{d (\mu \circ \varphi|_{U_i}^{-1})}{d\mu}(\varphi(x))$ 
for $\mu$-almost all $x\in U_i$, and all $i=1,...,n$.
\end{enumerate}
If (vi) holds, then the formula $\mu_{\varphi(x)}(\{x\})=\varrho(x)$ determines the continuous disintegration $\{\mu_x\}_{x\in X}$,
$L$ is given by \eqref{equ:transfer_operator_form2}
and $T^*h(y)=\sum_{x\in \varphi^{-1}(y)}\varrho(x) h(x)$,   $h\in L^2_\mu(X)$.
\end{prop}
\begin{proof}
\ref{enu:axiom_A1_concrete1}$\Rightarrow$\ref{enu:axiom_A1_concrete2}. If (A1) holds, then $L(a):=T^*aT$, $a\in A$, is a transfer operator for $\alpha$, and for any $a\in C(X)$ we have $\int_X L(a)(y) d\mu(y)=\langle 1, T^*aT1\rangle=\langle T 1, aT1\rangle= 
\int_X a(y) d\mu(y)$.  Hence $L^*(\mu)=\mu$. 

\ref{enu:axiom_A1_concrete2}$\Rightarrow$\ref{enu:axiom_A1_concrete3}. The transfer operator $L:C(X)\to C(X)$
 is of the form
\eqref{equ:transfer_operator_form} where $X\ni x \mapsto \mu_x \in \M(X)$ is continuous. 
The relation $L^*(\mu)=\mu$ means that we have
$
\int_X \Big(\int_{\varphi^{-1}(y)} a(x) d\mu_y(x)\Big) d\mu(y)= 
\int_X a(y) d\mu(y)$ for all $a\in C(X)$. Since $C(X)$ is dense in $L^1_\mu(X)$,  $\{\mu_x\}_{x\in X}$ is a disintegration of  $\mu$ 
along the fibers of $\varphi$.

\ref{enu:axiom_A1_concrete3}$\Rightarrow$\ref{enu:axiom_A1_concrete4}.  This immediately follows from \eqref{eq:concrete_adjoint_operator} and continuity of  $\{\mu_x\}_{x\in X}$.

\ref{enu:axiom_A1_concrete4}$\Rightarrow$\ref{enu:axiom_A1_concrete5}.  This is clear because $T$ preserves $C(X)\subseteq L^2_\mu(X)$.

\ref{enu:axiom_A1_concrete5}$\Rightarrow$\ref{enu:axiom_A1_concrete1}. The restriction $T:C(X)\to C(X)$ coincides with the endomorphism 
$\alpha:C(X)\to C(X)$. If the projection $TT^*$ preserves $C(X)$ it is a conditional expectation onto the range of $\alpha$. 
Hence $TT^*:C(X)\to C(X)$ is a $\alpha(C(X))$-bimodule map, and  for any $a\in C(X)$ there is $b\in C(X)$ with $TT^*a=\alpha(b)$. 
Thus for any $h\in C(X)\subseteq L^2_\mu(X)$, treating $a$ as an element of $A$, and using (A2) we get
$
TT^*aTh= \alpha(b) Th=T bh
$. 
Since  $C(X)$ is dense in $L^2_\mu(X)$, this shows that $TT^*AT\subseteq TA$.
Since $T^*T=1$, this gives $T^*AT\subseteq A$.

Thus conditions \ref{enu:axiom_A1_concrete1}-\ref{enu:axiom_A1_concrete5} are equivalent. If $\varphi$ is a local homeomorphism, 
then \ref{enu:axiom_A1_concrete5}$\Leftrightarrow$(vi)  by \cite[Proposition 4.2.1]{Przytycki}.
The remaining part is straightforward or discussed earlier.
\end{proof} 
\begin{rem}\label{rem:strong_Jacobian}
If $\varrho$  in Proposition \ref{prop:axiom_A1_concrete}(vi) exists and is strictly positive, then 
for every open $U\subseteq X$ on which
$\varphi$ is injective we have $\mu(\varphi(U))=\int_{U} \varrho(x)^{-1} d\mu(x)$. Hence $1/\varrho$ is 
a \emph{strong Jacobian} of $\varphi$  with
respect to $\mu$, see \cite[Definition 1.9.4]{Przytycki}. 
\end{rem}
\begin{ex}\label{ex:not_continous_dissintegration}
Let $X=[0,1)\,\,(mod\,\,1)$ and 
$\varphi(x)=2x   \,\,(mod\,\,1)$. For any $p_0, p_1 >0$ with $p_0+p_1=1$ define
 $\varrho=p_0 \mathds{1}_{[0,1/2)} + p_1 \mathds{1}_{[1/2,1)}$. 
 Then $\varrho>0$ is a cocycle for $\varphi$ but $\varrho$ is continuous if only if $p_0=p_1=\frac{1}{2}$.
Let $\lambda$ be the Lebesgue measure (length) on $[0,1)$ and put $\mu:= \varrho \lambda$, that is $\mu(U):=\int_U \varrho d\lambda$ for all Borel $
U \subseteq X$. Then 
 \(\mu \in \Inv(X,\varphi)\), $\supp \mu=X$, and the adjoint of the isometry $(Th)(x)=h(\varphi(x))$ is 
given by $
T^*h(y)=\sum_{x\in \varphi^{-1}(y)}\varrho(x) h(x)$,   $h\in L^2_\mu(X)$.
The corresponding disintegration $\{\mu_x\}_{x\in X}$   of  $\mu$ is given by $\mu_{x}=p_0\delta_{\frac{1}{2}x} + p_1\delta_{\frac{1}{2}(x+1)}$, $x\in X$.
The equivalent conditions in Proposition \ref{prop:axiom_A1_concrete} hold if and only if 
 $p_0=p_1=\frac{1}{2}$.
 \end{ex}
We may  reverse the above problem  and ask whether for a given dynamical system $(X,\varphi)$  and a transfer operator $L:C(X)\to C(X)$ there exists a $\varphi$-invariant measure
such that $\supp\mu= X$ and $L$ is implemented by the isometry \eqref{eq:concrete_composition_operator}.
Since $\supp \mu \subseteq \Omega(\varphi)$,  then we necessarily have   $\Omega(\varphi)=X$.
Using methods of the thermodynamical formalism we  give the following positive answer when $\varphi$ is expanding:
\begin{prop}\label{prop:cocycle_inducing_measures}
Let $L:C(X)\to C(X)$ be a transfer operator of finite type on  such that
the corresponding local homeomorphism $\varphi:X\to X$ is expanding 
and $\Omega(\varphi)=X$. Then there exists  $\mu\in \Inv(X,\varphi)$  with $\supp \mu= X$ such that 
idntyifying $C(X)$ with the algebra $A\subseteq \B(L^2_\mu(X) ) $ of multiplication operators we have
$$
L(a)=TaT^*, \qquad A=C(X),
$$
where  $T\in \B(L^2_\mu(X) )$ is the operator of composition with $\varphi$. That is, 
$aT$, $a\in A$, are well presented weighted shifts associated with  the transfer operator $L$.
If in addition $\varphi:X\to X$ is topologically transitive and the logarythm of the associated cocycle $\varrho:X\to (0,1]$ 
is H\"older continuous, then the measure $\mu$ above is unique and it is the Gibbs measure $\mu_{\varphi,\ln\varrho}$.
\end{prop}
\begin{proof}

Since $L(1)=1$, the dual operator $L^*:C(X)^*\to  C(X)^*$ preserves 
$\M(X) \subseteq C(X)^*$ which is a compact convex set in weak $*$-topology. 
By  the Schauder-Tychonoff fixed point theorem, there is 
$\mu\in \M(X)$ such that $L^*(\mu)=\mu$. This measure is $\varphi$-invariant as $\mu(f\circ \varphi)=\mu(L(f\circ \varphi))=
\mu(f)$ for $f\in C(X)$. In view of Proposition \ref{prop:axiom_A1_concrete}, it is left to show that $\supp \mu=X$.

Let us first consider the case when $\varphi:X\to X$ is topologically transitive.  Assume on the contrary that there is non-empty open $U$ with $\mu(U)=0$. Since the corresponding cocycle $\varrho:X\to (0,1]$ is strictly positive
this implies that $\mu(\varphi(U))=0$ by Remark \ref{rem:strong_Jacobian}. By induction we get that $\varphi^n(U)=0$ for all $n>0$. However, as $\varphi^k$ for some $k$ is topologicaly mixing, we 
get $X=\bigcup_{n=0}^{N}\varphi^n(U)$ for some $N$. This leads to a contradiction:  $1=\mu(X)\leq \sum_{n=0}^N\mu(\varphi^n(U))=0$. 
Thus $\mu$ is the desired measure by Proposition \ref{prop:axiom_A1_concrete}. If in addition $\ln\varrho$ is H\"older continuous, then 
 by Theorem \ref{thm:Ruelle-Perron-Frobenious},
the Gibbs measure $\mu=\mu_{\varphi,\ln\varrho}$ is the unique probability measure for which 
$L^*(\mu)=\mu$. 
 
For general $\varphi:X\to X$ we may use the spectral decomposition  $X=\bigsqcup_{j=1}^N X_j$  described in Theorem \ref{thm:spectral decomposition}.
We get the desired measure by putting $\mu:=\frac{1}{N}\sum_{j=1}^N \mu _j$
where $\mu_j$ is the measure constructed in the first step  for the topologically transitive map $\varphi|_{X_j}$ and the transfer operator given by cocycle $\varrho|_{X_j}$. 
\end{proof} 

\begin{ex} Let $(\Sigma_{\mathbb{A}},\sigma_{\mathbb{A}})$  be the topological Markov shift
 such that  all states $\{1,...,n\}$ are essential,  
see Example \ref{ex:TMS}.
Then $\Omega(\sigma_{\mathbb{A}})=\Sigma_{\mathbb{A}}$. By Proposition \ref{prop:cocycle_inducing_measures} 
for every cocycle  $\varrho:\Sigma_{\mathbb{A}}\to (0,1]$ for $\sigma_{\mathbb{A}}$
there is  $\mu\in \Inv(\Sigma_{\mathbb{A}},\sigma_{\mathbb{A}})$ 
such that 
$\int_{\Sigma_{\mathbb{A}}} \sum_{x\in \sigma_{\mathbb{A}}^{-1}(y)}\varrho(x) a(x) d\mu (y)
=\int_{\Sigma_{\mathbb{A}}} a(y) d\mu (y)$  for all $a\in C(\Sigma_{\mathbb{A}})$.
To say more about the relationship between $\varrho$ and $\mu$, recall that elements  of $\Inv(\Sigma_{\mathbb{A}},\sigma_{\mathbb{A}})$ with full support 
are in a one-to-one correspondence with  collections of nonnegative numbers $\{p_{\xi_1,\xi_2,...\xi_k}: (\xi_1,\xi_2,...\xi_k)\in\{1,...,n\}^{k}, k\in\Z\}$ such that $\sum_{i=1}^n p_i=1$,  
 $$
p_{\xi_1,\xi_2,...\xi_{k+1}}=\sum_{i=1}^n p_{\xi_1,\xi_2,...\xi_k,i}=\sum_{i=1}^n p_{i,\xi_2,\xi_3,...\xi_{k+1}},
$$
 and $p_{\xi_1,\xi_2,...\xi_{k}}>0$ if and only if $(\xi_1,\xi_2,...\xi_{k})$ is a path in the directed graph given by $\mathbb{A}$.
This correspondence is given by the equality $p_{\xi_1,\xi_2,...\xi_k}=\mu(C_{\xi_1,\xi_2,...\xi_k})$ 
where $C_{\xi_1,\xi_2,...\xi_k}\subseteq \Sigma_{\mathbb{A}}$ is the cylinder set of  sequences that start with \(\xi_1,\xi_2,...,\xi_k\).
Suppose that  $\varrho:\Sigma_{\mathbb{A}}\to (0,1]$  depends only on a finite number of coordinates, i.e. there is $N$ such that  $\varrho(\xi_1,\xi_2,...)=
\varrho(\xi_1,\xi_2,...,\xi_{N})$. Every such function  is H\"older continuous and $\varrho$ is induced by a measure $\mu$ if and only if 
$
\varrho(\xi_1,\xi_2,...,\xi_N)=\frac{p_{\xi_1,\xi_2,...,\xi_{N+k}}}{p_{\xi_2,...,\xi_{N+k}}}$ for all paths 
$(\xi_1,\xi_2,...\xi_{N+k})$.
If  $\mathbb{A}$ is irreducible, this measure is unique and it is  the Gibbs measure $\mu_{\sigma_{\mathbb{A}},\ln\varrho}$.
For instance, if  $\mathbb{A}$ is irreducible and $N=2$, then 
$\varrho=\sum_{i,j=1}^n p_{i,j}\mathds{1}_{C_{ij}}$
where $P=[p_{i,j}]_{i,j=1}^n$ is a left stochastic matrix. 
By the classical Perron-Frobenius theorem there exists a unique probability 
distribution $p=[p_i]_{i=1}^n$ such that $Pp=p$. In this case, the Gibbs measure is the \emph{Markov distribution} 
with the transition matrix $P$ and the initial 
distribution $p$:  $
\mu_{\sigma_{\mathbb{A}},\ln\varrho}(C_{\xi_1,...,\xi_k})=p_{\xi_1\xi_2}\cdot p_{\xi_2\xi_3}\cdot ... \cdot p_{\xi_{k-1}\xi_k}\cdot p_{\xi_k}.
$ 
Alternatively, we could consider the right stochastic matrix $Q=[q_{i,j}]_{i,j=1}^n$ where $q_{i,j}:=\frac{p_{i,j} p_j}{p_i}$. 
Then $
\mu_{\sigma_{\mathbb{A}},\ln\varrho}(C_{\xi_1,...,\xi_k})= p_{\xi_1}\cdot q_{\xi_1\xi_2}\cdot q_{\xi_2\xi_3}\cdot ... \cdot q_{\xi_{k-1}\xi_k}.
$ 
\end{ex}

\subsection{Weighted shifts associated with measure dynamical systems}
\label{subsection:Weighted shifts}
 Let us now consider a  situation where $(\Omega, \Sigma, \mu)$ is a $\sigma$-finite measure space and  $\Phi:\Omega\to \Omega$
 is a measurable map, with the measurable range \(\Phi(\Omega)\).
We are interested in  weighted operators of the form
\begin{equation}\label{eq:wso_measurable_isometry}
(Th)(\omega):=\rho(\omega)h(\Phi(\omega)), \qquad  h\in L^2_\mu(\Omega),
\end{equation}
where \(\rho:\Omega\to \C\) is a \(\Sigma\)-measurable function. Recall that $\Phi$ is \emph{$\mu$-almost surjective} 
if and only if $\mu(\Omega\setminus \Phi(\Omega))=0$ and
 then the measure $\mu\circ \Phi:\Phi^{-1}(\Sigma)\to [0,\infty]$ is defined. 
We write $\mu\circ \Phi\preceq \mu$ if $\mu\circ \Phi$ is absolutely continuous with
respect to $\mu$ restricted to \(\Phi^{-1}(\Sigma)\).
If $\mu(\Phi^{-1}(U))< \infty$ for every $U\in \Sigma$ with $\mu(U)<\infty$, we say that $\Phi$ is \emph{$\mu$-proper}.
If \(\mathcal{F}\) is a $\sigma$-subalgebra of \(\Sigma\),   the \emph{conditional expectation of 
 \(\rho\) with respect to \(\mathcal{F}\)} is 
an \(\mathcal{F}\)-measurable function  \(E(\rho|\mathcal{F}):\Omega\to \C\) such that 
$\int_U \rho d\mu =\int_{U} E(\rho|\mathcal{F}) d\mu$ for all $U\in \mathcal{F}$.
If it exists then \(E(\rho|\mathcal{F})\) is unique, up to equality $\mu$-almost everywhere. 

\begin{lem}\label{lem:wso_measurable_isometry}
Formula \eqref{eq:wso_measurable_isometry} defines an isometry on $L^2_\mu(\Omega)$ if and only if 
$\int_{\Phi^{-1}(U)} |\rho|^2 \, d\mu
=
\mu(U)$ for every $U\in \Sigma$. So if this is the case, then 
$\Phi$ is $\mu$-almost surjective,  $\mu\circ \Phi\preceq \mu$ 
and $|\rho|^2$ can be considered  a generalized Radon-Nikodym derivative of $\mu\circ \Phi$ with respect to 
$\mu$. 

More specifically, if $\Phi$ is $\mu$-almost surjective,  $\mu\circ \Phi\preceq \mu$ and $\frac{d\mu\circ \Phi}{d\mu}$
exists, which is always the case when $\Phi$ is $\mu$-proper, then \eqref{eq:wso_measurable_isometry} defines an isometry
if and only if $
E(|\rho|^2|\Phi^{-1}(\Sigma)) =\frac{d\mu\circ \Phi}{d\mu}
$, so in particular, one can always put  \(\rho:=\sqrt{\frac{d\mu\circ \Phi}{d\mu}}\).
\end{lem} 
\begin{proof}
Plainly, $T$ is a well defined isometry if and only if for every $U\in \Sigma$ we have
$\mu(U)=\|T\mathds{1}_U\|^2= \int |\rho|^2  (\mathds{1}_U\circ \Phi) \, d\mu=\int_{\Phi^{-1}(U)} |\rho|^2 \, d\mu $. 
If this holds, then $\mu(\Omega\setminus \Phi(\Omega))=\int_{\emptyset} |\rho|^2 \, d\mu =0$, that is
$\Phi$ is $\mu$-almost surjective. Moreover, since $\int_{\Phi^{-1}(U)} |\rho|^2 \, d\mu
=
\mu(U)=(\mu\circ \Phi)(\Phi^{-1}(U))$, we see  that $\mu\circ \Phi\preceq \mu$ on $\Phi^{-1}(\Sigma)$. 
Thus if $\frac{d\mu\circ \Phi}{d\mu}$ exists  then
$|\rho|^2=\frac{d\mu\circ \Phi}{d\mu} 
$ $\mu|_{\Phi^{-1}(\Sigma)}$-almost everywhere, and hence $
E(|\rho|^2|\Phi^{-1}(\Sigma)) =\frac{d\mu\circ \Phi}{d\mu}
$. When $\Phi$ is $\mu$-proper, then $\mu\circ \Phi$ 
is $\sigma$-finite and therefore $\frac{d\mu\circ \Phi}{d\mu}$ exists.
 \end{proof}
\begin{rem}\label{rem:composition_op_measure} We recall, cf. e.g. \cite{Ridge}, that  the formula 
\begin{equation}\label{eq:wso_measurable_composition}
(T_{\Phi}h)(\omega)=h(\Phi(\omega)), \qquad  h\in L^2_\mu(\Omega),
\end{equation}
defines  the bounded operator  $T_{\Phi}:L^2_\mu(\Omega)\to L^2_\mu(\Omega)$ 
if and only if $\mu\circ \Phi^{-1}\preceq \mu$ and  the Radon-Nikodym derivative
$\frac{d(\mu\circ \Phi^{-1})}{d\mu}
$ is essentially bounded. Then  $\|T_{\Phi}\|=\|\frac{d(\mu\circ \Phi^{-1})}{d\mu}\|_{L^{\infty}_{\mu}(\Omega) }^{1/2}$.  
\end{rem}

We identify the $C^*$-algebra $A:=L^\infty_\mu(\Omega)$
of essentially bounded functions with the algebra of  operators of multiplication on $L^2_\mu(\Omega)$:
$$
(ah)(\omega)=a(\omega)h(\omega), \qquad a\in A=L^\infty_\mu(\Omega),\,\, h\in L^2_\mu(\Omega).
$$
Thus if \eqref{eq:wso_measurable_composition} defines a bounded operator, then all  weighted composition operators $aT_\Phi$, $a\in A$, are bounded.
The formula $\alpha(a)=a\circ \Phi$, $a\in A$,  
gives a well defined map $\alpha:L^\infty_\mu(\Omega)\to L^\infty_\mu(\Omega)$ if and only if  $\mu\circ \Phi^{-1}\preceq \mu$, and 
then $\alpha$ is  a unital endomorphism.
\begin{prop}\label{prop:abstract_wso_from_measure_system}
Assume that \eqref{eq:wso_measurable_isometry} defines an isometry $T\in \B(L^2_\mu(\Omega))$, 
cf. Lemma \ref{lem:wso_measurable_isometry}. Assume also that  
 $\alpha(a)=a\circ \Phi$ defines an endomorphism of $A:=L^\infty_\mu(\Omega)$, 
that is $\mu\circ \Phi^{-1}\preceq \mu$. 
Then axioms (A1), (A2) are satisfied, so $aT$, $a\in A$, are abstract weighted shift operators. 
Moreover, 
\begin{enumerate}
\item (A3) holds if and only if $\rho \neq 0$ $\mu$-almost everywhere.

\item Every weighted composition operator $aT_\Phi$, $a\in A$, is an abstract weighted composition operator
$bT$, $b\in A$, if and only if  $\ess\inf |\rho|> 0$. 
\end{enumerate}
\end{prop}
\begin{proof} It is immediate that $Ta=\alpha(a)T$,  for all $a\in A$, which is axiom (A2). Now using (A2) for all $a,b\in A$ we have 
$$
T^*aT b =T^*a\alpha(b)T=  T^*\alpha(b)aT= b T^*aT.
$$
Thus $T^*aT\in A'=(L^{\infty}_{\mu}(\Omega))'=L^{\infty}_{\mu}(\Omega)=A$ (here we use the well known fact that 
 $L^{\infty}_{\mu}(\Omega)$ is a maximal abelian subalgebra of $\B(L^2_\mu(\Omega))$).
Hence (A1) holds. One readily sees that for $a\in A$ we have
$
(aT)(L^2_\mu(\Omega))=L^2_\mu(\{\omega\in \Omega: a(\omega)\rho(\omega) \neq 0\}).
$
Hence $\overline{ATL^2_\mu(\Omega)}=L^2_\mu(\Omega)$ if and only if  $\rho \neq 0$ $\mu$-almost everywhere. This proves (i). 
Note that   $\ess\inf |\rho| >0$ if and only if $\rho^{-1}$ exists and belongs to $L^{\infty}_{\mu}(\Omega)$.
Thus if this holds,  then for every $a\in A$, we have $b:=a \rho^{-1}\in A$ and $aT_{\Phi}=b T$. 
Conversely, if $T_\Phi=aT$ for some $a\in A$, then $a\rho=1$  $\mu$-almost everywhere. 
That is, $\rho^{-1}=a\in A=L^\infty_\mu(\Omega)$ and therefore 
$\ess\inf |\rho| >0$. This proves (ii).
\end{proof}
\begin{cor}\label{cor:abstract_wso_from_measure_system}
Suppose that $\Phi$ is $\mu$-almost surjective, $\mu$-proper,  $\mu\circ \Phi\preceq \mu$ and $\mu\circ \Phi^{-1}\preceq \mu$. 
Then $T:=\sqrt{\frac{d\mu\circ \Phi}{d\mu}}T_{\Phi}$ is an isometry,  $(L^{\infty}_{\mu}(\Omega),T)$ satisfies  (A1), (A2) and 
\begin{enumerate}
\item (A3) holds if and only if $\mu$ is forward quasi-invariant with respect to $\Phi$, in the sense 
that $\frac{d\mu\circ \Phi}{d\mu}>0$ $\mu$-almost everywhere.

\item Every weighted composition operator $aT_\Phi$, $a\in A$, is an abstract weighted composition operator
$bT$, $b\in A$, if and only if   $\mu$ is forward quasi-invariant in the strong sense that  
$(\frac{d\mu\circ \Phi}{d\mu})^{-1}\in L^{\infty}_{\mu}(\Omega)$. 
\end{enumerate}
\end{cor}
\begin{proof} Apply  Proposition \ref{prop:abstract_wso_from_measure_system} to \(\rho:=\sqrt{\frac{d\mu\circ \Phi}{d\mu}}\), which is allowed   by Lemma \ref{lem:wso_measurable_isometry}.
\end{proof}

Now we specialize  to the case where \(\Omega=V\) is  a countable set
and \(\mu\) is the counting measure. So the initial data is just a map \(\Phi:V\to V\). 
Then $\mu$-surjectivity is just surjectivity, and if $\frac{d\mu\circ \Phi}{d\mu}$ exists
then \(\Phi\) is finite-to-one and $\frac{d\mu\circ \Phi}{d\mu}(v)=|\Phi^{-1}(v)|$, for $v\in V$. Thus we get the following:
\begin{cor}\label{cor:weighted_shifts_discrete}
Let $\Phi:V\to V$ and $\rho:V\to \C$ be \textbf{} two maps where $V$ is a set. 
The formula 
$(Th)(v):=\rho(v)h(\Phi(v))$ defines an isometry on $H:=\ell^2(V) $ if and only if
 \(\Phi:V\to V\) is a surjective map and 
$\sum_{w\in \Phi^{-1}(v)}|\rho(w)|^2=1$ for every  $v\in V.
$
Assume $T$ is an isometry and let $A\cong\ell^\infty(V)$ be the algebra of operators of mutiplication on $H$. 
Then axioms (A1), (A2) are satisfied and hence operators of
the form $aT$, $a\in A$, are abstract weighted shift operators. Axiom (A3) holds if and only if $\rho >0$.
\end{cor}
\begin{rem}\label{rem:weighted_shifts_on_trees}
Corollary \ref{cor:weighted_shifts_discrete} covers a large class of bounded weighted shifts on directed trees introduced and analyzed in \cite{Stochel}. 
\end{rem}

\section{Spectral radius}\label{sec:Spectral radius}
In this section we discuss variational formulas for the spectral radii of abstract weighted shift operators. 
The main new result is Theorem \ref{thm:Ruelle's}.
The basic fact is that our investigations can be reduced to the study of  spectral radius of weighted transfer operators:
\begin{lem}\label{lem:first_spectral_radius}
Let $A\subseteq \B(H)$ be  a unital commutative $C^*$-subalgebra and let $T\in \B(H)$ be an isometry.
Assume (A1), i.e. $T^*AT\subseteq A$. 
The spectral radius of the bounded operator $aT:H\to H$, $a\in A$, is equal to the square root of the spectral radius
of the  operator $L|a|^2:A\to A$ where $L|a|^2(b):=T^*|a|^2bT$, $b\in A$. That, is $
r(aT)=\sqrt{r(L|a|^2)}$, for $a\in A.
$
 \end{lem}
\begin{proof} Let   $n\in \N$. Since $(L|a|^2)^n:A\to A$ is a positive map, $\|(L|a|^2)^n\|_{\B(A)}=\|(L|a|^2)^n(1)\|_{A}$, 
see for instance \cite[Lemma 2.1]{kwa_Exel}. Thus by the $C^*$-equality and the definition of $L|a|^2$, we get 
$
\|(a T)^n\|^2=\|(a T)^{*n} (a T)^n\|=\|(L|a|^2)^n(1)\|=\|(L|a|^2)^n\|.
$
Hence $r(aT)=\lim_{n\to \infty} \|(a T)^n\|^{1/n}=\lim_{n\to \infty} \|(L|a|^2)^n\|^{1/2n}=\sqrt{r((L|a|^2)^n)}$.
\end{proof}

The spectral radius of a general weighted  transfer operator was described in \cite{t-entropy} using the following notion 
 of  $t$-entropy (here we recall 
a  definition  simplified in  \cite{t-entropy2}).
 \begin{defn}[\cite{t-entropy}, \cite{t-entropy2}] Let $L:C(X)\to C(X)$ be a (unital) transfer operator for an endomorphism  whose 
dual map is 
$\varphi:X\to X$.  The \emph{$t$-entropy of  $L$} is the functional $\tau_{L}:\Inv(X,\varphi)\to \R\cup \{-\infty\}$  whose value at $\mu\in
\Inv(X,\varphi)$ is defined by the  formula:
\begin{equation}\label{equ:t-entropy}
 \tau_{L}(\mu)  := \inf_{n\in\mathbb N}\frac{1}{n} \inf_{D\in D(X)} \sum_{g\in D}\mu(g) \ln\frac{\mu(L^ng)}{\mu(g)}
\end{equation}
where $D(X)$ is the set of all continuous partitions of unity. 
If
we have $\mu(g)= 0$ for a certain function $g\in D$, then we set the corresponding summand
in~\eqref{equ:t-entropy} to be zero. If there exists 
$g\in D(X)$ such that  $L^ng\equiv  0$ and  $\mu(g)>0$, then we set $\tau_{L}(\mu)  =
-\infty$.
\end{defn}
If \(L:C(X)\to C(X)\) is a transfer operator and $b\in C(X)^+$ is a positive function, then by \cite[Theorem 11.2]{t-entropy}
we have
$
\ln r(Lb)=\max_{\mu\in \Inv(X,\varphi)} \left(\int_{{X}}\ln\vert b\vert\,d\mu +\tau_{L}(\mu)\right)
$
where $\tau_{L}(\mu)$ is the $t$-entropy for $L$. In particular, $\ln r(L)=\max_{\mu\in \Inv(X,\varphi)}\tau_{L}(\mu)$.
\begin{thm}[Antonevich-Bakthin-Lebedev]\label{thm:Antonevich-Bakthin-Lebedev}
Let $aT$, $a\in A$, be an abstract weighted shift associated with a transfer operator $L$  (i.e. assume (A1) and (A2))
where $A\cong C(X)$. Then
$$
\ln r(aT)=\max_{\mu\in \Inv(X,\varphi)} \left(\int_{{X}}\ln\vert a\vert\,d\mu
+\frac{\tau_{L}(\mu)}{2}\right), 
$$
where $\varphi:X\to X$ is the dual map to the endomorphism in (A2) and $\tau_{L}$ is  the $t$-entropy for $L$ (we adapt the convention  that $\ln 0=-\infty$).
\end{thm}
\begin{proof} Combine  Lemma \ref{lem:first_spectral_radius} and \cite[Theorem 11.2]{t-entropy}. Or apply \cite[Theorem 13.6]{t-entropy}.
\end{proof}
\begin{rem}
If $\varphi:X\to X$ is a homeomorphism, then we are in the situation of Theorem \ref{thm:lebedev},  see Proposition \ref{prop:T_unitary}.
Hence $\tau_{L}(\mu)=0$ for every $\mu\in \Inv(X,\varphi)$, by \cite[Proposition 8.3]{t-entropy}. Thus the above formula agrees 
with \eqref{eq:spectral_radius_invertible}.
\end{rem}

When $\varphi$ is an expanding open map, 
then the $t$-entropy $\tau_{L}(\mu)$ of $L$ is closely related to  
the Kolmogorov-Sinai entropy $h_{\varphi}(\mu)$. This issue was not pursued in \cite{t-entropy}, \cite{t-entropy2}.
We will now clarify this relationship  and  extend the last part of Theorem \ref{thm:Ruelle-Perron-Frobenious} to 
arbitrary expanding open maps $\varphi$ and arbitrary continuous weights $c \geq 0$.

\subsection{General Ruelle's formula for expanding local homeomorphisms}

Let us fix a local homeomorphism $\varphi:X\to X$ on a compact metric space and a continuous function  $c:X\to [0,\infty)$. 
Let $\mathcal{L}_{c}:C(X)\to C(X)$ be the  Ruelle-Perron-Frobenius transfer operator given by  \eqref{eq:Ruelle-Perron_Frobenious}.
By Theorem \ref{thm:Ruelle-Perron-Frobenious}, if $c$ is strictly positive and H\"older continuous,  and $\varphi$ is expanding and topologically transitive,
 then $\ln r(\LL_{c})=\sup_{\mu\in \Inv(X,\varphi)} \left(\int_{{X}}\ln c\,d\mu +h_{\varphi}(\mu)\right)$.
In  we now show that this formula holds for arbitrary continuous $c$ and arbitrary expanding open $\varphi$. The proof is independent of Theorem \ref{thm:Ruelle-Perron-Frobenious}.

Using the  equalities $\|\mathcal{L}_{c}^n\|=\|\mathcal{L}_{ c}^n(1)\|$, $n\in \N$, and the Gelfand formula for  spectral radius, in general we have
\begin{equation}\label{eq:spectral_radius_Ruelle-Perron_Frobenious}
\ln r(\mathcal{L}_c)=\lim_{n\to \infty} \max_{y\in X}\frac{1}{n} \ln\sum_{x\in \varphi^{-n}(y)}\prod_{i=0}^{n-1} c(\varphi^i (x)).
\end{equation}
The following lemma is a part of \cite[Theorem 12]{Lebedev_Maslak}, which is stated without a proof and therefore   we give a short argument. 
\begin{lem}\label{lem:inequality_local_homeo}
If $\varphi:X\to X$ is a local homeomorphism and $c:X\to (0,\infty)$ is strictly positive, then
\begin{equation}\label{eq:inequality_spectral_radius_Ruelle-Perron_Frobenious}
\ln r(\mathcal{L}_c) \leq P(\varphi, \ln c)=\sup_{\mu\in \Inv(X,\varphi)} \left(\int_{{X}}\ln c\,d\mu +h_{\varphi}(\mu)\right).
\end{equation}
\end{lem}
\begin{proof}
In view of formulas \eqref{eq:topological_pressure} and \eqref{eq:spectral_radius_Ruelle-Perron_Frobenious}
 it suffices to show that
 there exists $\varepsilon > 0$ such that for each $n\in \N$ and $x\in X$ 
the set $\varphi^{-n} (x)$ is  $\varepsilon$-separated
in  $n$-Bowen's metric $d_n$.
The  latter holds for every local homeomorphism.    Indeed, the set $\Delta:=\{x:|\varphi^{-1}(x)|=1\}$ is clopen in $X$. 
The function $d:X\setminus \Delta\to (0,\infty)$ given by $d(x)=\min\{d(y_1,y_2):y_1,y_2\in \varphi^{-1}(x), y_1\neq y_2\}>0$ is continuous on the compact set $X\setminus \Delta$. 
Thus 
$
\varepsilon:=\frac{1}{2}\min \left\{\min_{x\in X\setminus \Delta} d(x), d(X\setminus \Delta, \Delta)\right\}>0
$ is non-zero.
It is straightforward to check that this is the desired number. 
\end{proof}
In general, if $\varphi$ is not expanding, inequality \eqref{eq:inequality_spectral_radius_Ruelle-Perron_Frobenious}
might be sharp. Indeed, if $\varphi$ is a homeomorphism and $c\equiv 1$, then $\ln r(\mathcal{L}_c)=0$ and $P(\varphi, 0)=h(\varphi)$ may be an arbitrary non-negative number, cf. 
\cite[Section 7.3]{Walt2}. 
\begin{thm}\label{thm:Ruelle's} Suppose that $\varphi:X\to X$ is an expanding open map on a metrizable compact space $X$.
   For any continuous $c:X\to [0,\infty)$  we have
$$
\ln r(\mathcal{L}_c)=\max_{\mu\in \Inv(X,\varphi)} \left(\int_{{X}}\ln c\,d\mu +h_{\varphi}(\mu)\right)
=\max_{\mu\in \Erg(X,\varphi)} \left(\int_{{X}}\ln c\,d\mu +h_{\varphi}(\mu)\right),
$$ 
where $h_{\varphi}(\mu)$ is the Kolmogorov-Sinai entropy  ($h_{\varphi}(\mu)<\infty$ for all $\mu\in \Inv(X,\varphi)$). 
\end{thm}
\begin{proof} Since  $\varphi$ is expanding the entropy map $\Inv(X,\varphi)\ni \mu \to h_{\varphi}(\mu)\in [0,\infty)$ is upper semi-continuous, see \cite[Theorem 3.5.6]{Przytycki}.
The map $\Inv(X,\varphi)\ni \mu \to \int_{{X}}\ln c\,d\mu$ is always upper semi-continuous, cf. the proof of
 \cite[Lemma 1.4]{Ruelle}. Therefore the supremum $\sup_{\mu\in \Inv(X,\varphi)} \left(\int_{{X}}\ln c\,d\mu +h_{\varphi}(\mu)\right)$
is in fact a maximum. In addition, using the ergodic decomposition we have 
$$
\max_{\mu\in \Inv(X,\varphi)} \left(\int_{{X}}\ln c\,d\mu +h_{\varphi}(\mu)\right)
=\max_{\mu\in \Erg(X,\varphi)} \left(\int_{{X}}\ln c\,d\mu +h_{\varphi}(\mu)\right),
$$
cf. \cite[Corollary 2.4.3]{Przytycki} or \cite[9.10.1]{Walt2}.

Our proof of the assertion goes by a consecutive relaxing of assumptions.  First we assume that $c>0$ is strictly positive.
Then, in view of Lemma \ref{lem:inequality_local_homeo}, we need to show that $P(\varphi,\ln c)\leq \ln r(\mathcal{L}_c)$.

1) Assume first that $\varphi:X\to X$ is topologically transitive. By Theorem \ref{thm:spectral decomposition}, 
$X=\bigsqcup_{j=1}^k X_j$ is a union of $k$ disjoint clopen sets $X_j$  which are cyclically
permuted by $\varphi$ and such that $\varphi^{k}:X_j\to X_j$ is topologically mixing. 
Hence by \cite[Proposition 2]{Fan_Jiang} for every $j=1,...,k$ and every $\varepsilon>0$ there is $p_j(\varepsilon)\in\N$ such that
$\varphi^{kp_j(\varepsilon)}(B(x,\varepsilon)\cap X_j))=X_j$ for every $x\in X_j$.  
Thus there is $p(\varepsilon)\geq \max_{j=1,...k}p_j(\varepsilon)$ such that $\varphi^{p(\varepsilon)}(B(x,\varepsilon)\cap X_j))=X_j$ 
for every $x\in X_j$ and $j=1,...,k$. By \cite[Proposition 1.5]{Fan_Jiang} for sufficiently small 
$\varepsilon>0$  we have $\varphi^n(B_n(x,\varepsilon))=B(\varphi^n(x),\varepsilon)$ for all $n\in \N$ and $x\in X$ 
(here $B_n(x,\varepsilon)$ is a ball in Bowen's $n$-metric). Accordingly, for all $x\in X_i$, $i=1,...,k$, and $n\in \N$ we get
$$
\varphi^{n+p(\varepsilon)}(B_n(x,\varepsilon))=\varphi^{p(\varepsilon)}(B(\varphi^n(x),\varepsilon))\supseteq X_i\,\, \text{ if }\,\,\varphi^n(x)\in X_i. 
$$ 
This implies that for any elements $y_i \in X_i$, $i=1,...,k$, the set $\varphi^{-(n+p(\varepsilon))}(\{y_1,...,y_k\})$ is $(d_n,\varepsilon)$-spanning. 
Now  recall, cf. \cite[9.3]{Walt2}, 
 that apart from \eqref{eq:topological_pressure} the topological pressure
is also given by 
$$
P(\varphi,\ln c)=\lim_{\varepsilon \to 0} \limsup_{n\to\infty}\inf_{E\subseteq X \text{ is }
\atop
(d_n,\varepsilon)\text{-spanning}} \frac{1}{n} \ln\sum_{y\in E}\prod_{i=0}^{n-1} c(\varphi^i (x)).
$$
For any finite set $E\subseteq X$ and any function $b:X\to [0,\infty)$ we have 
$$
\sum_{x\in E} b(x)\leq \sum_{x\in \varphi^{-n}(\varphi^n(E))} b(x)= \sum_{y\in \varphi^n(E)}\sum_{x\in \varphi^{-n}(y)}b(x) 
\leq |\varphi^n(E)| \sup_{y\in X}\sum_{x\in \varphi^{-n}(y)} b(x),
$$
which by taking $b(x)=\prod_{i=0}^{n-1} c(\varphi^i (x))$ gives
$$
 \frac{1}{n} \ln\sum_{x\in E}\prod_{i=0}^{n-1} c(\varphi^i (x))\leq \max_{y\in X}\frac{1}{n} \ln\sum_{x\in \varphi^{-1}(y)}\prod_{i=0}^{n-1} c(\varphi^i (x))
+ \ln|\varphi^n(E)|^{\frac{1}{n}}.
$$
Thus, in view of \eqref{eq:spectral_radius_Ruelle-Perron_Frobenious} and the above formula for $P(\varphi,\ln c)$,
  it suffices to 
show  that  for sufficiently small $\varepsilon>0 $ we have
$$\limsup_{n\to\infty}\inf_{E\subseteq X \text{ is }
\atop
(d_n,\varepsilon)\text{-spanning}}  \ln|\varphi^n(E)|^{\frac{1}{n}}=0.
$$
But taking $E:=\varphi^{-(n+p(\varepsilon))}(\{y_1,...,y_k\})$ to be  the $(d_n,\varepsilon)$-spanning  set  
constructed above we get 
$$
  \ln|\varphi^n(E)|^{\frac{1}{n}}= \ln |\varphi^n(\varphi^{-(n+p(\varepsilon))}(\{y_1,...,y_k\})) |^{\frac{1}{n}} = \ln  |\varphi^{-p(\varepsilon)}(\{y_1,...,y_k\})|^{\frac{1}{n}},
$$
which tends to zero, as $n\to \infty$. 

2) Now suppose that  the set of periodic points is dense in $X$ (equivalently $\Omega(\varphi)=X$).
By the spectral decomposition (Theorem \ref{thm:spectral decomposition}) we  have $X=\bigsqcup_{j=1}^N X_j$ where $X_j=\varphi^{-1}(X_j)$ are compact and the maps $\varphi|_{X_j}:X_j\to X_j$ are topologically transitive (and expanding open). Clearly,
  $\mathcal{L}_c=\bigoplus_{j=1}^N \mathcal{L}_{ c|_{X_j}}$ and $\Erg(X,\varphi)=\bigsqcup_{j=1}^N\Erg(X_j,\varphi|_{X_j})$.
 Hence by step 1) we get
$$
 r(\mathcal{L}_c)=\max_{j=1,...,N} r(\mathcal{L}_{c|_{X_j}})= \max_{\mu \in \Erg(X_j,\varphi) \atop j=1,...,N} e^{\int_{X_j}\ln c\,d\mu +h_{\varphi}(\mu)}
=\max_{\mu \in \Erg(X,\varphi)} e^{\int_{{X}}\ln c\,d\mu +h_{\varphi}(\mu)},
$$
which gives the assertion in this case.

3) Let  $\varphi:X\to X$ be an arbitrary expanding open map. The restriction of $\varphi$ to $\Omega(\varphi)=\overline{\text{Per}(\varphi)}$ is also expanding and it remains open 
by \cite[Lemma 3.3.10]{Przytycki}. Hence we may consider the transfer operator 
$\mathcal{L}_{c|_{\Omega(\varphi)}}:C(\Omega(\varphi))\to C(\Omega(\varphi))$ 
associated to $\varphi:\Omega(\varphi)\to \Omega(\varphi)$.
Since $\|\mathcal{L}_{c|_{\Omega(\varphi)}}^n(1)\|\leq \|\mathcal{L}_c^n(1)\|$, 
we have $r( \mathcal{L}_{c|_{\Omega(\varphi)}})\leq r( \mathcal{L}_c)$.
 Thus using step 2) and that every $\mu \in  \Inv(X,\varphi)$ is supported on $\Omega(\varphi)$  
 we get
$$
\max_{\mu\in \Inv(X,\varphi)} e^{\int_{{X}}\ln c\,d\mu +h_{\varphi}(\mu)}=\max_{\mu\in \Inv(\Omega(\varphi),\varphi)} e^{\int_{\Omega(\varphi)}\ln c\,d\mu +h_{\varphi}(\mu)}
=r( \mathcal{L}_{c|_{\Omega(\varphi)}})\leq r( \mathcal{L}_c),
$$
which is the inequality we wanted to prove.

 Now let us consider the case when $c$ is arbitrary (may have zeros). Choose  strictly positive continuous functions $c_n>0$ such that $c_n\searrow c$. 
Then $\{r(\mathcal{L}_{c_n})\}_{n=1}^\infty$ is a decreasing sequence with values not greater than $r(\mathcal{L}_c)$. 
Hence by the upper semi-continuity of spectral radius we get $r(\mathcal{L}_{c_n})\searrow r(\mathcal{L}_c)$.
Also putting $P:=\max_{\mu\in \Inv(X,\varphi)} \left(\int_{{X}}\ln c\,d\mu +h_{\varphi}(\mu)\right)$ we have  
  $\max_{\mu\in \Inv(X,\varphi)} \left(\int_{{X}}\ln c_n\,d\mu +h_{\varphi}(\mu)\right)\searrow P$. Indeed, the sets 
$A_n:=\{\mu: \int_{{X}}\ln c_n\,d\mu +h_{\varphi}(\mu) \geq P\}$ are  non-empty and descending. They are compact, as the maps  $\Inv(X,\varphi)\ni \mu \mapsto \int_{{X}}\ln c_n\,d\mu +h_{\varphi}(\mu)$
are upper semi-continuous. Hence  $\bigcap_{n=1}^\infty A_n\neq \emptyset $ and for any 
$\mu_0 \in \bigcap_{n=1}^\infty A_n$ we have  $P= \int_{{X}}\ln c\,d\mu_0 +h(\mu_0)$. Thus using step 3) we get
$$
 r(\mathcal{L}_c)= \lim_{n\to \infty } r(\mathcal{L}_{\ln c_n})=\lim_{n\to \infty}\max_{\mu\in \Inv(X,\varphi)} e^{\int_{X}\ln c_n\,d\mu +h_{\varphi}(\mu)}=\max_{\mu\in \Inv(X,\varphi)} e^{\int_{X}\ln c\,d\mu +h_{\varphi}(\mu)}.
$$
\end{proof}

\begin{cor}\label{cor:strictly_positive_formuals}
Suppose that $\varphi:X\to X$ is an  expanding open map on a metrizable compact space 
$X$ and let $c:X\to [0,\infty)$ be a continuous map such that $c|_{\Omega(\varphi)}> 0$ and   $\ln c|_{\Omega(\varphi)}$ is H\"older continuous. 
Let $\Omega(\varphi)=\bigsqcup_{j=1}^N X_j$
be the spectral decomposition for $(\Omega(\varphi),\varphi)$. Then for each pair of restrictions $\varphi|_{X_j}$ and $\ln c|_{X_j}$ there is the unique   Gibbs measure $\mu_{j}$ and we have
$$
\ln r(\mathcal{L}_c)=\max_{j=1,...,N} \int_{X_j}\ln c\,d\mu_j +h(\mu_j)=\max_{j=1,...,N} \lim_{n\to \infty} \frac{1}{n} \ln\sum_{x\in \varphi^{-n}(y_j)}\prod_{i=0}^{n-1} c(\varphi^i (x))
$$
where $y_j\in X_j$ are arbitrary points for $j=1,...,N$.  In particular, if  $\varphi:\Omega(\varphi)\to \Omega(\varphi)$ is topologically transitive, then 
there is a unique measure $\mu\in  \Inv(X,\varphi)$, the Gibbs measure for $\varphi|_{\Omega(\varphi)}$ and $\ln c|_{\Omega(\varphi)}$, for which we have
$$
\ln r(\mathcal{L}_c)= \int_{X}\ln c\,d\mu +h_{\varphi}(\mu)=\lim_{n\to \infty} \frac{1}{n} \ln\sum_{x\in \varphi^{-n}(y)}\prod_{i=0}^{n-1} c(\varphi^i (x)) \text{ for all }y\in \Omega(\varphi).
$$
\end{cor}
\begin{proof}
By (the proof of) Theorem \ref{thm:Ruelle's} we have $\ln r(\mathcal{L}_c)=\max_{j=1,...,N} P(\varphi|_{X_j},\ln c|_{X_j})$.  
Combining this with Theorem \ref{thm:existence_of_Gibbs_measures}
we get the assertion.
\end{proof}
\begin{cor}
\label{cor:abstract_spectral_radius_expanding}
Let $aT$, $a\in A$, be  abstract weighted shifts associated with a transfer operator $L$  (i.e. assume (A1) and (A2))
where $A\cong C(X)$ is  separable. Suppose also that  the map $\varphi:X\to X$ dual to the endomorphism in (A2) is  expanding and open. Then
$$
\ln r(aT)=\max_{\mu\in \Erg(X,\varphi)} \int_{{X}}\ln\vert a\sqrt{\varrho}\vert\,d\mu
+\frac{h_{\varphi}(\mu)}{2}, \qquad a\in A,
$$
where $\varrho$ is the cocycle associated to  $L$, cf. \eqref{equ:transfer_operator_form2}.
If in addition $|a|^2\varrho |_{\Omega(\varphi)}>0$ 
and its logarithm is H\"older continuous then for each system $(X_j, \varphi|_{X_j})$ where $\Omega(\varphi)=\bigsqcup_{j=1}^N X_j$
is the spectral decomposition for $(\Omega(\varphi),\varphi)$, there is the Gibbs measure $\mu_j$ and then 
$$
\ln r(aT)=\max_{j=1,...,N} \int_{{X}}\ln\vert a\sqrt{\varrho}\vert\,d\mu_j
+\frac{h_{\varphi}(\mu_j)}{2}.
$$
In particular, if $\varphi:\Omega(\varphi)\to \Omega(\varphi)$ is also topologically transitive, then 
the Gibbs measure $\mu$ for $\varphi|_{\Omega(\varphi)}$ with potential $\ln (|a|^2\varrho )|_{\Omega(\varphi)}$ is the unique measure in $ \Inv(X,\varphi)$ for which
$
\ln r(aT)=\int_{{X}}\ln\vert a\sqrt{\varrho}\vert\,d\mu
+\frac{h_{\varphi}(\mu)}{2}.
$
\end{cor}
\begin{proof} By Lemma \ref{lem:first_spectral_radius}, $r(aT)=r(L|a|^2)^{1/2}=  r(\mathcal{L}_{|a|^2\varrho})^{1/2}$. Hence by Theorem \ref{thm:Ruelle's},
$
\ln r(aT)={1/2}\ln r(\mathcal{L}_{|a|^2\varrho})=
\max_{\mu\in \Erg(X,\varphi)} \max_{\mu\in \Erg(X,\varphi)} \int_{{X}}\ln\vert a\sqrt{\varrho}\vert\,d\mu
+\frac{h_{\varphi}(\mu)}{2}.
$
To get  the second part of the assertion apply Corollary \ref{cor:strictly_positive_formuals}. 
\end{proof}
 The variational formula in Corollary \ref{cor:abstract_spectral_radius_expanding} could be also deduced from 
Theorem \ref{thm:Antonevich-Bakthin-Lebedev}, as we have the following 
relationship between the $t$-entropy and the Kolmogorov-Sinai entropy.
\begin{cor}
\label{cor:t-entropy vs Kolmogorov-Sinai entropy} 
Let $L:C(X)\to C(X)$ be a transfer operator for a unital injective endomorphism whose dual map  
$\varphi:X\to X$ is expanding and open. Then for each $\mu\in
\Inv(X,\varphi)$ the $t$-entropy $\tau_{L}(\mu)$ of $L$ is given by 
$$
\tau_{L}(\mu)=\int_{{X}}\ln \varrho \,d\mu +h_{\varphi}(\mu),
$$
where $\varrho$ is the cocycle associated to  $L$ and $h_{\varphi}(\mu)$  is the Kolmogorov-Sinai entropy of  $\mu$.
In particular, for every $\mu\in
\Inv(X,\varphi)$ we have 
$$
h_{\varphi}(\mu)= \inf_{b\in C(X,\R)} \ln r(L e^{b}) - \int_X\ln (e^b\varrho)\, d\mu, 
$$
where $L e^{b}:C(X)\to C(X)$ is the weighted transfer operator $Le^{b}(a):=L(e^{b}a)$, $a\in C(X)$.
\end{cor}
\begin{proof}
By \cite[Proposition 2.2]{t-entropy} the functional $\lambda: C(X,\R) \ni b\mapsto \ln r(Le^{b})=\ln r(\mathcal{L}_{b\ln \varrho})\in \R\cup \{-\infty\}$ is convex and continuous. Hence
its Legendre transform $\lambda^*:C(X,\R)^*\to \R\cup \{+ \infty\}$, given by $\lambda (\mu^*)= \sup_{b\in C(X,\R)} \mu(b)-\lambda(b)$, 
is a  unique  convex and lower semicontinuous functional on $C(X)^*$ such that
$$
\lambda (b)= \sup_{\mu\in C(X)^*} \mu(b)-\lambda^*(b), \qquad b\in C(X,\R),
$$
cf. \cite[Proposition 3.1]{t-entropy}. The authors of \cite{t-entropy} call  $S:=-\lambda^*$  the dual entropy map for $L$. 
By  \cite[Proposition 4.2]{t-entropy} the effective domain of $\lambda^*$ is contained 
in $\Inv(X,\varphi)$. Hence $\lambda^*$ is determined by its values on $\Inv(X,\varphi)$. 
The map $\Inv(X,\varphi)\ni \mu \to \int_{{X}}\ln \varrho \,d\mu +h_{\varphi}(\mu)$  is upper semicontinuous and affine, cf. \cite[Theorem 8.1]{Walt2}. 
By Theorem \ref{thm:Ruelle's} we have $\lambda (b)= \max_{\mu\in C(X)^*} \mu(b)+ \int_{{X}}\ln \varrho \,d\mu +h_{\varphi}(\mu)$. 
Hence $S(\mu)=-\lambda^*(\mu)=\int_{{X}}\ln \varrho \,d\mu +h_{\varphi}(\mu)$.  On the other hand, we have  $\tau_{L}= -\lambda^*|_{\Inv(X,\varphi) }$ by \cite[Theorem 5.6]{t-entropy}.
This gives the assertion. 
\end{proof}
\section{$C^*$-algebras generated by abstract weighted shift operators}\label{sec:C*-algebras}

In this section we discuss conditions under which the $C^*$-algebra $C^*(A,T)=C^*(\{aT:a\in A\})$ generated by the abstract weighted shift operators $aT$, $a\in A$, is  modeled  
by a crossed product $A\rtimes L$, which is a special example of a Cuntz-Pimsner algebra.
This allows us to
employ some recent results from the general theory of $C^*$-algebras associated to $C^*$-correspondences.
In this way we get natural dynamical criteria for rotational invariance of the spectrum $\sigma(aT)$, $a\in A$, and 
for the coincidence of $\sigma(aT)$ with the essential spectrum $\sigma_{ess}(aT)$ (Corollary \ref{cor:Spectral_consequences}).
We will also 
explain when $A\subseteq C^*(A,T)$ is a Cartan $C^*$-subalgebra (Theorem \ref{thm:Cartan_crossed_products}), which will be crucial
in our study of Riesz projectors and description of spectra.

The basic algebraic structure inferred from axioms (A1), (A2) is the following: 
\begin{lem} Suppose $A\subseteq \B(H)$ is a  unital $C^*$-algebra and $T\in \B(H)$ is an  isometry satisfying (A1) and (A2). 
So that we have an endomorphism $\alpha:A\to A$ and the associated transfer operator $L:A\to A$ satisfying $L(a)=T^*aT$, $a\in A$. For  $n,m,k,l\in \N_0$ and $a,b,c,d\in A$ we have 
\begin{equation}\label{eq:commutation relations}
(aT^n T^{*m}b)\cdot  (cT^k T^{*l}d)
=\begin{cases}
aT^n T^{*m-k+l}\alpha^l(L^k(bc)) d & m\geq k,
\\
a\alpha^n(L^m(bc))T^{k-m+n} T^{*l} d & m< k.
\end{cases}
\end{equation}
In particular,
 $\spane\{aT^n T^{*m}b: a,b\in A, n,m\in \N\}$ is a dense $*$-subalgebra of $C^*(A, T)$. 
Moreover, $\spane\{aT^n T^{*n}b: a,b\in A, n\in \N\}$ is a  $*$-subalgebra of $C^*(A, T)$ and hence its closure
$$
B:=\clsp\{aT^n T^{*n}b: a,b\in A, n\in \N\}
$$
is a  $C^*$-subalgebra of $C^*(A, T)=\clsp\{aT^n T^{*m}b: a,b\in A, n,m\in \N\}$.
\end{lem}
\begin{proof}  Relation \eqref{eq:commutation relations} readily follows from the transfer identity \eqref{eq:transfer_operator_relation} 
and the commutation relation (A2).
 By the definition 
$\spane\{aT^n T^{*m}b: a,b\in A, n,m\in \N\}$ and $\spane\{aT^n T^{*n}b: a,b\in A, n\in \N\}$ are linear spaces which are invariant under involution. By 
\eqref{eq:commutation relations} these spaces are also closed under multiplication. This implies the assertion. 
\end{proof}
\begin{defn}
We call the $C^*$-algebra $B$ defined above the \emph{core subalgebra} of $C^*(A, T)$.
\end{defn}

\subsection{The crossed product and Cuntz-Pimsner picture of $C^*(A,T)$}
In this subsection we fix a  (unital) transfer operator $L:A\to A$ on a unital commutative $C^*$-algebra $A$. 
We recall the construction of the crossed product $A\rtimes L $ of $A$ by $L$ from \cite{kwa_Exel}. We will assume that the transfer operator $L:A\to A$ is \emph{faithful}. 
By Proposition \ref{prop:faithfulness_of_transfers}, in the present paper, we are interested  only in this case. 
Then there is a unique endomorphism $\alpha:A\to A$ for which $L$ is a transfer operator. 
Also then $A\rtimes L $ coincides with the crossed product originally introduced by Exel in \cite{exel3}, see \cite[Proposition 4.9]{kwa_Exel}.

The \emph{Toeplitz algebra of $(A,L)$} is the universal $C^*$-algebra $\TT(A,L):=C^*(A,t)$ generated by the $C^*$-algebra $A$ and an element $t$ subject to relations
$$
L(a)=t^*at, \qquad a\in A. 
$$
Then $t^*t=L(1)=1$, so $t$ is an isometry, and we have  $ta=\alpha(a)t$, for $a\in A$, by \cite[Proposition 4.3]{kwa_Exel}. 
We define  a \emph{redundancy} to be a pair $(a,k)$ where $ a\in A$ and $ k\in \overline{Att^*A}$ are such that 
$$
abt=kbt, \qquad \textrm{ for all }  b\in A.
$$
Let $\RR$ be the ideal in $\TT(A,L)$ generated by the set $
\{a-k: (a,k) \textrm{ is a redundancy}\}. 
$
\begin{defn}
\emph{The crossed product}    of $A$ by the  faithful transfer operator $L$ is the  quotient $C^*$-algebra  $A\rtimes L:=\TT(A,L)/\RR$. 
Since $A\cap \RR=\{0\}$ we identify $A$ with its image in $A\rtimes L $. The image of $t$ in $A\rtimes L$ is denoted by $s$. 
Hence  $A\rtimes L =C^*(A,s)$ is generated by $A$ and the isometry $s$.
\end{defn}
\begin{lem}\label{lem:epimorphism}
Suppose that $aT\in \B(H)$, $a\in A\subseteq \B(H)$, are well presented abstract weighted shifts associated with a transfer operator $L:A\to A$.  Then there is  a $*$-epimorphism 
$$
\Psi:A\rtimes L\to C^*(A,T) \quad \text{ where 
 $\quad \Psi|_A=id$ and $\Psi(s)=T$.}
$$
\end{lem}
\begin{proof}
Universality of $\TT(A,L)$ gives a $*$-epimorphism $
\widetilde{\Psi}:\TT(A,L)\to C^*(A,T)$  where  $\widetilde{\Psi}|_A=id$ and $\widetilde{\Psi}(t)=T$.
 Axiom (A3) implies that $\RR\subseteq \ker \widetilde{\Psi}$. Indeed, if $(a,k)$ is a redundancy, then for all 
$b\in A$ and $h\in H$ we have
$
abTh = \widetilde{\Psi}(abt)h=\widetilde{\Psi}(kbt)h=\widetilde{\Psi}(k)bTh,
$
which by (A3) means that $a=\widetilde{\Psi}(k)$, that is $\widetilde{\Psi}(a-k)=0$. Thus $\widetilde{\Psi}$ factors through to 
the desired epimorphism $\Psi$.
\end{proof}
To study  faithfulness of $\Psi$ we use the isomorphism $A\rtimes L\cong \OO_{M_L}$ of the crossed product with the Cuntz-Pimsner algebra of a 
 $C^*$-correspondence $M_L$ defined as follows, see \cite{br}, \cite{kwa_Exel}. 
We  make $A$ into an inner product (right) $A$-module  by putting 
$
m\cdot a :=m\al(a)$ and $\langle m,n \rangle_{L}:=L(m^*n)$,  for all $n,m,a \in A$. 
The inner product is non-degenerate because we assumed that $L$ is  faithful.   
Completing $A$ in the norm $\|n\|_{L}=\sqrt{\|\langle n,n \rangle_{L}\|}$ we get a Hilbert $A$-module $M_{L}$.  
For each $a\in A$ we have an adjointable operator $\phi(a)\in \LL(M_L)$ where 
$
\phi(a) m:= am$, $m\in A\subseteq M_L$. 
Thus $M_{L}$ is a $C^*$-correspondence over $A$ with an injective and nondegenerate left action $\phi:A\to \LL(M_L)$. 
 \begin{thm}\label{thm:gauge_uniqueness} Suppose $L:A\to A$ is a faithful transfer operator.
 The crossed product $A\rtimes L$ is canonically isomorphic to the Cuntz-Pimsner algebra $\OO_{M_L}$.  
If  $aT\in \B(H)$, $a\in A\subseteq \B(H)$, are well presented abstract weighted shifts that  generate the transfer operator $L:A\to A$, then the $*$-epimorphism 
$
\Psi:A\rtimes L\to C^*(A,T)$ in Lemma \ref{lem:epimorphism}  restricts to an isomorphism of core subalgebras
$$
\clsp\{as^n s^{*n}b: a,b\in A, n\in \N\}\cong \clsp\{aT^n T^{*n}b: a,b\in A, n\in \N\}
$$
and the following conditions are equivalent:
\begin{enumerate}
\item\label{it:gauge_uniqueness1} $\Psi$ is an isomorphism: $A\rtimes L\cong C^*(A,T)$;
\item\label{it:gauge_uniqueness2} $C^*(A,T)$ is equipped with a circle gauge-action, i.e. there is a group homomorphism $\gamma:\T\to \Aut (C^*(A, T))$ such that $\gamma_z|_A=id$ and  $\gamma_z(T)= zT$ for $z\in \T$;
\item\label{it:gauge_uniqueness3} There is a conditional expectation $E:C^*(A,T)\to B\subseteq C^*(A,T)$  that annihilates all the operators of the form $T^mT^{*n}$ with $n\neq m$.
 Equivalently, for all finite $F\subseteq \N$ and $n_k,m_k\in \N$, $a_k,b_k\in A$ for $k\in F$ we have  $ \left\|\sum\limits_{k\in F, n_k=m_k} a_kT^{n_k}T^{*m_k}b_k\right\| \leq \left\|\sum_{k\in F} a_kT^{n_k}T^{*m_k}b_k\right\| $ . 
\end{enumerate}
Moreover, there is always a faithful representation $\Psi:A\rtimes L\to B(H)$ such that putting $T:=\Psi(s)$ and identifying $A$ with $\Psi(A)$ the axioms (A1), (A2), (A3) hold;   that is 
the operators $aT$, are well presented abstract weighted shifts associated with $L:A\to A$. 
\end{thm}

\begin{proof}
The isomorphism $A\rtimes L\cong \mathcal{O}_{M_L}$ is a known fact, see  
\cite[Theorem 4.7]{kwa_Exel} or \cite[Proposition 3.10]{br}. Under this isomorphism we may identify $M_L$ with the closure of $As$ in $A\rtimes L$.
The core subalgebra of $\mathcal{O}_{M_L}$ which is generated by elements in $M_L^n M_L^{*n}$, $n\in \N$, coincides with the core subalgebra
$\clsp\{as^n s^{*n}b: a,b\in A, n\in \N\}$ of $A\rtimes L=C^*(A,s)$. 
Every representation of $\mathcal{O}_{M_L}\cong A\rtimes L$ which is faithful on $A$ is faithful on the core subalgebra, 
see, for instance, \cite[Proposition 6.3]{katsura}.
This gives the isomorphism of $\Psi$ on the core subalgebra. 
Faithfulness of a representation on the whole algebra  $\mathcal{O}_{M_L}\cong A\rtimes L$
is known to be equivalent to  existence of the canonical circle action or the canonical condition expectation onto the core subalgebra,
 cf. \cite[Theorem 6.4]{katsura}
or  \cite[Lemma 6.2]{CKO}. This translates to  equivalence of conditions \ref{it:gauge_uniqueness1}-\ref{it:gauge_uniqueness3}.

The left action on the $C^*$-correspondence $M_L$ is given by a unital monomorphism. Hence this action  is injective and nondegenerate. Thus the main result of  \cite{Hirshberg}
gives  a faithful representation $\Psi:\mathcal{O}_{M_L}\to \B(H)$ such that $\overline{\Psi(M_L)}H=H$. 
Identifying $\mathcal{O}_{M_L}$ with $A\rtimes  L$, $A$ with $\Psi(A)$, and putting  $T:=\Psi(s)$
the relation $\overline{\Psi(M_L)}H=H$ translates to (A3). Axioms (A1) and (A2) hold  because they hold in $A\rtimes L$ (and in fact in $\TT(A,L)$).
\end{proof}
\begin{cor}\label{cor:rotation_invariance} If the equivalent conditions in Theorem \ref{thm:gauge_uniqueness} hold, then for every $a\in A$,   $\sigma(aT)$ is invariant under rotations around zero.

\end{cor}
\begin{proof}
Using the action  $\gamma:\T\to \Aut (C^*(A, T))$ we get  for every $\lambda \in \C$ and $z\in\T$  that
$\gamma_{\overline{z}}(aT -\lambda 1)= a\overline{z}T -\lambda 1= \overline{z}(aT -z\lambda 1)$.
Hence $\lambda \in \sigma(aT)$ if and only if $z\lambda \in \sigma(aT)$.
\end{proof}
In the second part of Theorem \ref{thm:gauge_uniqueness}  we applied the main result of \cite{Hirshberg} 
on representations of Cuntz-Pimsner algebras. Applying the main result of \cite{CKO} 
we get that the equivalent conditions in Theorem \ref{thm:gauge_uniqueness} hold automatically when $\varphi$ is topologically free:
\begin{defn}\label{def:topological freeness} 
We say that a continuous map $\varphi:X\to X$ is \emph{topologically free} if  the set of periodic points $\{x\in X:  \varphi^n(x)=x \text{ for some }n>0\}$ has empty interior
in $X$. 
\end{defn}
\begin{thm}\label{thm:isomorphism}
Suppose that $aT\in \B(H)$, $a\in A\subseteq \B(H)$, are well presented abstract weighted shifts. Assume that  the associated dual map $\varphi:X\to X$ is topologically free.  Then the $*$-epimorphism in Lemma \ref{lem:epimorphism} is always an isomorphism $A\rtimes L\cong C^*(A,T)$.
In fact, every representation of $A\rtimes L$ which is injective on $A$ is faithful. 
\end{thm}
\begin{proof}
As explained in \cite[Example 9.14]{CKO} topological freeness of $\varphi$  implies topological freeness of the $C^*$-correspondence $M_L$ (in the sense of \cite[Definition 7.1]{CKO}). Hence the result follows from \cite[Theorem 1.1]{CKO}.
\end{proof}

\begin{cor}\label{cor:Spectral_consequences}
Retain the assumptions of Theorem \ref{thm:isomorphism}. Then  $\sigma(aT)$, $a\in A$, is invariant under rotations around zero.
If in addition 
$A$ contains no non-zero compact operators, which is automatic  when $X$ has no isolated points,
 then $\sigma_{ess}(aT)=\sigma(aT)$.
\end{cor}
\begin{proof} The first part is a consequence of Theorem \ref{thm:isomorphism} and Corollary \ref{cor:rotation_invariance}.
For the second part it suffices to show that whenever $A\cap \K(H)=\{0\}$ then $C^*(A,T)$ embeds 
 into the Calkin algebra $\B(H)/\K(H)$  via the quotient map $q:\B(H)\to \B(H)/\K(H)$.
 This  follows from Theorem \ref{thm:isomorphism}.  If $0\neq a\in A\cap \K(H)$,
then $a(X)=\sigma(a)$  has a non-zero isolated point $z_0$ which is an eigenvalue for $a$ with finite multiplicity.
This implies that $a^{-1}(z_0)\subseteq X$ consists of isolated points.
\end{proof}

\subsection{The crossed product and its groupoid picture in the finite type case}
We briefly discuss the case when $L:C(X)\to C(X)$ is a \emph{transfer operator of finite type}.
Then the crossed product $C(X)\rtimes L$ has a groupoid model which has been extensively studied, and 
the result of previous subsection can be recovered from this groupoid picture, cf. \cite{exel_vershik}.  
Topological freeness has an important characterisation that will be crucial in our study of Riesz projectors.
 
Recall that, by Proposition \ref{prop:finite_type_characterisation}, see also Corollary \ref{cor:transfers_of_finite_type},
the associated map $\varphi:X\to X$ is a surjective local homeomorphism, and $L$ is given by a continuous cocycle    
$\varrho:X\to (0,1]$ with strictly positive values.
Let $\{v_i\}_{i=1}^n\subseteq C(X)$ be the partition of unity  subordinated to the open cover $\{U_i\}_{i=1}^n$ of $X$,
 where $\varphi|_{U_i}$ is injective. 
Put
$u_i(x)=\sqrt{\frac{v_i(x)}{\varrho(x)}}$ for $i=1,...,n$.
Then the  \emph{crossed product} $A\rtimes L$ can be defined as the universal $C^*$-algebra generated by a copy of the $C^*$-algebra $A=C(X)$ and 
an isometry $s$ subject to relations
$$
s^*as=L(a), \text{ for all }a\in A,\quad \text{and}\quad \sum_{i=1}^n u_iss^* u_i=1.
$$
This is how Exel's crossed product is described in \cite{exel_vershik}, modulo  \cite[Corollary 7.2]{exel_vershik}. 
If  $aT\in \B(H)$, $a\in A\subseteq \B(H)$, are well presented abstract weighted shifts associated with  $L:C(X)\to C(X)$,
then the above universal description of $A\rtimes L$ and 
 Lemma \ref{lem:A3_characterisation_for_finite_type} give a canonical epimorphism $A\rtimes L \mapsto C^*(A,T)$.
In particular, 
Theorems \ref{thm:gauge_uniqueness} and  \ref{thm:isomorphism}
can be deduced  from \cite[Theorems 4.2 and 10.3]{exel_vershik} in this case. 

There is a standard groupoid $\G$  associated to 
 the local homeomorphism $\varphi:X\to X$. It is often called the \emph{Renault-Deaconu groupoid}, as  for  topological Bernoulli shifts  it was 
 studied by Renault \cite[page 140]{Renault}, and in general it was popularized by Deaconu \cite{Deaconu}. Such groupoids were also considered by  Arzumanian and Vershik \cite{Arzu_Vershik2} and a number of other authors. 
We put
$
\G:= \{(x,n,y)\in X\rtimes \Z\rtimes X: \exists_{k,l\in\N} \,\, n=k-l \textrm{ and }\varphi^k(x)=\varphi^l(y)\}
$
and equip it with algebraic operations given by
$
(x,n,y) (y,m,z):= (x,n+m,z)$,  $(x,n,y)^{-1}:=(y,-n,x).
$
We identify $X$ with the space of units of $\G$, via the map $x\mapsto (x,0,x)$, and then the range and source maps are
$
r(x,n, y) =x$, $s(x, n, y) =y$. 
The topology on $\G$  has a basis consisting of sets of the form
$
Z(U,k,l,V):=\{(x,k-l,y): x\in U,\, y\in V \textrm{ and } \varphi^k(x)=\varphi^l(y)\}
$
where $k,l \in \N$, and  $U,V\subseteq X$ are open sets such that $\varphi^{k}|_U$ and  $\varphi^{l}|_V$ are homeomorphisms with open ranges. 
This makes  $\G$  a locally compact \'etale Hausdorff groupoid. Recall that the $C^*$-algebra $C^*(\G)$ of $\G$  is defined as 
the maximal $C^*$-completion of the $*$-algebra $C_c(\G)$ of continuous and compactly supported functions on $G$ with
the standard convolution multiplication and $*$-operation:
$
(f g)(\gamma)=\sum_{\gamma_1\gamma_2=\gamma} f(\gamma_1)g(\gamma_2)$ and  $f^*(\gamma)=\overline{f(\gamma^{-1})}$ for $\gamma,\gamma_1, \gamma_2\in \G$ and $f,g\in C_c(\G)$.
One can show that we have 
$$
C^*(\G)\cong C(X)\rtimes L
$$ 
where the isomorphism is identity on $C(X)\subseteq C_c(\G)$ and maps the function $v\in C_c(\G)$ given by
$
v(x,1,\varphi(x))={\varrho(\varphi(x))}^{\frac{1}{2}}$   and  $v(x,n,y)=0$ for $n\neq 1$,
to the isometry $s\in C(X)\rtimes L$, cf. \cite[1.3.4]{Anantharaman-Delaroche} and \cite[Theorem 9.1]{exel_vershik}. Also it is well known that
$C^*(\G)$ coincides with the reduced $C^*$-algebra associated to $\G$, 
and thus we have a faithful conditional expectation from  $C(X)\rtimes L$ onto $C(X)$,
see \cite{Renault_Cuntz_like}. 
Using the above isomorphism we characterise when $C(X)$ is a maximal abelian subalgebra of $C(X)\rtimes L$, and hence a Cartan subalgebra  in the sense of \cite{Renault}. 
This slightly extends  a part of the main result of \cite{CS} proved by different methods.

\begin{thm}
\label{thm:Cartan_crossed_products}
Let $L:C(X)\to C(X)$ be a transfer operator of finite type. Then $C(X)$ is a
maximal abelian subalgebra of $C(X)\rtimes L$ if and only if the associated dual map $\varphi:X\to X$ is topologically free. 
\end{thm}
\begin{proof}
Our notion of topological freeness of $\varphi$ is equivalent to the one used in  \cite{exel_vershik}, \cite{CS} (and to essential principality of $\varphi$ in \cite{Deaconu}).
The latter is known, and easily seen, to be equivalent to effectiveness (as well as to topological freeness and topological principality) of the associated groupoid $\G$, see \cite{Deaconu} or  \cite[Lemma 2.21 and Corollary 2.24]{KwaMeyer0}. 
The inclusion $C(X)\subseteq C^*(\G)$ is a Cartan subalgebra if and only if $\G$ is effective by \cite[Corollary 7.6]{KwaMeyer} (which generalizes \cite[Theorem 4.2]{Re} to non-separable case).
\end{proof}
\section{When Riesz projectors correspond to invariant sets}\label{sec:Riesz_projectors}
Let $A\subseteq \B(H)$ be a unital commutative $C^*$-subalgebra of $\B(H)$. Assume the identification $A=C(X)$. For every open set $U\subseteq X$ we denote by $\mathds{1}_{U}\in \B(H)$ the orthogonal projection 
that supports the ideal $C_0(U)\subseteq C(X)$. That is, $\mathds{1}_{U}$ is the strong limit of a an approximate unit in $C_0(U)$:
$
\mathds{1}_{U}=s\text{-}\lim \mu_{\lambda} 
$
where $\{\mu_\lambda\}_{\lambda}\in C_0(U)$ are functions such that $0\leq \mu_\lambda \leq 1$. 
One of the aims of this section is to prove the following:
\begin{prop}\label{prop:Riesz_projector} Suppose that  $aT$, $a\in A\cong C(X)$,  are well presented abstract weighted shifts,  and that the dual map $\varphi:X\to X$ is topologically free, or more generally that $C^*(A,T)$ is canonically isomorphic to $A\rtimes L$ for  the associated transfer operator $L:A\to A$. 
Assume that $a\in A=C(X)$ 
 is invertible  or more generally  $\mathds{1}_{X\setminus a^{-1}(0)}=1$.
Then every Riesz projector corresponding to a clopen part of $\sigma(aT)$
 belongs to $A=C(X)$ and  $P=\mathds{1}_{X_0}\in C(X)$  where ${X_0}\subseteq X$ is a clopen set with  $\varphi^{-1}({X_0})={X_0}$.

\end{prop}
 When $T$ is unitary the above assertion
holds for every (not necessarily invertible) $a\in A$, see \cite[Theorem 7.2]{Anton_Lebed}.
But when $T$ is not invertible this assertion fails without the  assumption $\mathds{1}_{X\setminus a^{-1}(0)}=1$:
\begin{ex}\label{ex:contrexample}
Let $V=\{v_n:n\in \Z\}\sqcup \{w_{-n}:n\in \N\}$ be a set indexed by $\Z\sqcup \N$. 
Let $\Phi:V\to V$ be given by 
$
\Phi(v_n)=v_{n+1}$ and $\Phi(w_n)=\begin{cases} 
w_{n+1}  & n<-1
\\
v_0 & n=-1
\end{cases}
$.
The weighted composition operators  $(aT_{\Phi}h)(v):= a(v)h(\Phi(v))$, $a\in \ell^{\infty}(V)$, acting on $H=\ell^2(V)$ are bounded (these are weighted shifts on a directed tree \cite{Stochel}). 
They are also
 well  presented (abstract) weighted shifts
$aT$, $a\in A:=\ell^{\infty}(V)$, where $(Th)(v):=\sqrt{|\Phi^{-1}(v)|}h(\Phi(v))$, cf. Corollary \ref{cor:weighted_shifts_discrete}.
The associated dual map is topologically free, cf. 
Lemma \ref{lem:characterisation_topol_freeness} below.  Taking $a\in A$ such that $a(v_n):=2$ for $n\in \Z$,
 $a(w_{-n})=1/2$ for $n>1$ and $a(w_{-1})=0$, we find that 
$$
aT_{\Phi}=2U_\Z\oplus \frac{1}{2} T_{\N}
$$ where $U_\Z$ and $T_{\N}$  are  the classical two-sided shift and one-sided shift, respectively. Hence 
$\sigma(aT_{\Phi})= \{z\in \C: |z|\leq \frac{1}{2} \text{ or }|z|=2\}$ does not intersect $S^1$. The corresponding Riesz projector
$P$  is the operator of multiplication by $\mathds{1}_{\{w_{-n}: n>1\}}$. Hence $P\in A$. But the set $\{w_{-n}: n>1\}$ is not $\Phi$-invariant
 and therefore the corresponding set $U\subseteq X$ is not invariant under the  map $\varphi$ induced on the spectrum of $A$.
\end{ex}
Retain the assumptions of Proposition \ref{prop:Riesz_projector}.
In view of Corollary \ref{cor:rotation_invariance}, 
$\sigma(aT)$ is invariant under rotations around zero. Therefore every Riesz projector is of the form $P_1$ or $P_2-P_1$ where 
$P_1$, $P_2$ are Riesz projectors corresponding to  parts of $\sigma(aT)$ lying in disks centered at zero. 
Thus the analysis in this section boils down to considering the case when $aT$ is \emph{hyperbolic}, i.e. $\sigma(aT)\cap S^1=\emptyset$, and $P$ is the Riesz projector
corresponding to the part of $\sigma(aT)$ lying inside the disk $\{z\in \C:|z|<1\}$.

For any hyperbolic operator $b\in \B(H)$  
the formula
$$ 
P:=\frac{1}{2\pi i}\int_{S^1}(\lambda-b)^{-1}d\lambda \in \B(H),
$$
defines a Riesz projector,  cf. \cite{Riesz}.
The spaces  $H_1=PH$ and  $H_2=\ker P$ are complementary, but might not be orthogonal. They are preserved by $b$ and 
$
\sigma(b|_{H_1})\subseteq\{z\in\mathbb{C}: |z|<1\}$ and $ 
\sigma(b|_{H_2})\subseteq\{z\in\mathbb{C}: |z|>1\}.
$
Hence there is $n_0\in\N$ and 
$\varepsilon>0$ such that
\begin{equation}\label{a1}
\|b^n h_1\| < (1-\varepsilon)^n\|h_1\|\qquad \text{for all } h_1\in H_1, \,\,n>n_0,
\end{equation}
\begin{equation}\label{a2}
\|b^n h_2\| > (1+\varepsilon)^n\|h_2\|\qquad  \text{for all } h_2\in H_2, \,\,n>n_0.
\end{equation}
Every $h \in H$ has a unique decomposition $h=h_1+h_2$ where
$h_i\in H_i$, $i=1,2$. Since $\|b^n h\|=\|b^n(h_1+h_2)\| \geqslant \|b^n h_2\|-\|b^n
h_1\|,$ we see that \eqref{a1}, \eqref{a2} imply that
either $ \lim_{n\rightarrow\infty} \|b^n h\|=\infty$
or $ \lim_{n\rightarrow\infty} \|b^n h\|=0$ depending on whether  $h_2\neq 0$ or $h_2=0$.
Accordingly, 
\begin{equation}\label{a3}
PH= \{h\in  H:\lim_{n\rightarrow\infty}b^n h=0\}.
\end{equation}


\begin{lem}\label{lem:Riesz_projector} 
Assume conditions (A1), (A2) and that $a\in A \cong C(X)$ is such that $\sigma(aT)\cap S^1=\emptyset$. Let
$ P=\frac{1}{2\pi i}\int_{S^1}(\lambda-aT)^{-1}d\lambda
$ and put $H_1=PH$ and $H_2=\ker P=(1-P)H$. Then
\begin{enumerate}
\item \label{enu:Riesz_projection1} $P\in A'$, i.e. $P$  commutes with elements of $A=C(X)$;
\item\label{enu:Riesz_projection1.5}  $|a|T$ commutes with $P$;
\item\label{enu:Riesz_projection2}  $a|_{H_2}:H_2\to H_2$ is invertible;
\item\label{enu:Riesz_projection3}  $TH_1\subseteq H_1$ and $\mathds{1}_{X\setminus a^{-1}(0)}\,TH_2\subseteq H_2$.
In particular,  $P$ commutes with  $\mathds{1}_{X\setminus a^{-1}(0)}\,T$. Hence if $\mathds{1}_{X\setminus a^{-1}(0)}=1$, then $T$ commutes with $P$.
\end{enumerate}

\end{lem}

\begin{proof} \ref{enu:Riesz_projection1}. By \eqref{a3},  $H_1= \{h\in  H:\lim_{n\rightarrow\infty}(aT)^n h=0\}$. For any 
$b\in A$ and $h\in
H_1$ 
$$\|(aT)^n b h\|=\|\alpha^n(b)(aT)^n h\|\leqslant \|b\|\|(aT)^n
h\|\, \mathop{\longrightarrow}_{n\rightarrow\infty}\,0.
$$
Thus $bH_1\subseteq H_1$ for all $b\in A$. Assume that $b\in A$ is invertible. For any $h\in H_2$ we have $bh=h_1 + h_2$ 
where $h_i\in H_i$, $i=1,2$. Thus $h=b^{-1}h_1+h_2$. Since $b^{-1}h_1 \in H_1$ and $h-h_2\in H_2$, we must have $h_1=0$.
Hence $bH_2\subseteq H_2$ for all invertible $b\in A$. This implies that $bH_2\subseteq H_2$ for all  $b\in A$, 
as every element in $A$ is a sum of four invertible elements. This proves that $P\in A'$.
 
\ref{enu:Riesz_projection1.5}.  Using the Tietze theorem we may find functions $a_k\in
 C(X)$, $k=1,2,...$, such that:
$|a_k(x)|\leqslant 1$  and $a_k(x)=\frac{|a(x)|}{a(x)}$  when $|a(x)|\geqslant \frac{1}{k}.
$
Then $\lim_{k\rightarrow\infty}a_k a= |a|$. Since $P$ commutes with $aT$
and $a_k\in C(X)$ we get
$$
(|a|T) P = \lim_{k\rightarrow\infty}a_k aT P=P \lim_{k\rightarrow\infty}a_k aT = P (|a|T).
$$

\ref{enu:Riesz_projection2}. Using that $aT:H_2\to H_2$ is surjective and $a$ commutes with $(1-P)$ we get
$$ 
 H_2=aTH_2=aT (1-P)H= (1-P)aT H=a(1-P)T H \subseteq aH_2.
$$ 
Hence $a:H_2\to H_2$ is surjective. To see it is injective, 
it suffices to show that $|a|:H_2\to H_2$ is injective, 
because $a^*a=|a|^2$. Suppose that $h\in H_2$ is such that $|a|h=0$. 
By the Urysohn lemma we may find functions  $\varphi_k\in C(X)$,  $k=2,3,...$, such that
$$
0\leqslant \varphi_k\leqslant 1,\qquad \varphi_k(x)=
\begin{cases}
 0, & \text{when}\,\, |a(x)|\leqslant\frac{1}{k}, \\
1,& \text{when}\,\, |a(x)|\geqslant\frac{1}{k-1}.
\end{cases}
$$
Let us put 
$
|a|_k(x):=|a(x)|$, when  $|a(x)|>\frac{1}{k}$,
and $
|a|_k(x):=\frac{1}{k}$, when  $|a(x)|\leqslant \frac{1}{k}$.
Then   $|a|_k\in C(X)$ are invertible elements in $C(X)$ such that $|a|_k\varphi_k=|a|\varphi_k$.  
Thus
$
|a|_k(\varphi_k h)=\varphi_k(|a| h)=0
$. This implies that  $\varphi_k h=0$, for all  $k=2,3,...\,$. 
Since the restriction of $aT$ to $H_2$ is invertible  there is $h_2\in H_2$ such that $aT h_2= h$. From the construction of $\varphi_k$ we have
$\lim_{k\rightarrow\infty}\varphi_k a= a.$
Accordingly, we get
$$
 h=aT h_2=\lim_{k\rightarrow\infty}  \varphi_k aT h_2=\lim_{k\rightarrow\infty} \varphi_k h=0.
$$

\ref{enu:Riesz_projection3}. Let $h\in H_1$ and let $Th=h_1+h_2$, where  $h_i\in H_i$, $i=1,2$. Then $(aT)h=ah_1+ah_2\subseteq H_1$ and hence  $ah_2=0$.
By \ref{enu:Riesz_projection2}, $a$ is invertible on $H_2$ and thus $h_2=0$. This  proves that 
$
TH_1\subseteq H_1.
$

To prove that $1_{X\setminus a^{-1}(0)}TH_2\subseteq H_2$, we need to show that $uTH_2\subseteq H_2$ for any $u\in C_0(X\setminus a^{-1}(0))\subseteq C(X)$.  However, for every $\varepsilon >0 $ the open set $\{x\in X: |u(x)|<\varepsilon\}$ contains the compact set
$a^{-1}(0)=\bigcap_{n\in \N}\{x \in X: |a(x)|<\frac{1}{k}\}$ and hence it contains all the sets $\{x \in X: |a(x)|<\frac{1}{k}\}$
for sufficiently large $k$. Thus  if $\varphi_k$ are as in the proof of item \ref{enu:Riesz_projection2}, 
then $\|\varphi_k u-u\|< \varepsilon$ for sufficiently large $k$. That is,  $\varphi_ku$ converges to $u$. 
Therefore  the proof reduces to showing
that $\varphi_kTH_2\subseteq H_2$ for sufficiently large $k$.
However, as in item \ref{enu:Riesz_projection1.5}
 we may find functions  $a_k\in
 C(X)$, for $k=1,2,...$, such that
   $a_k(x)=1/a(x)$  when 
	$|a(x)|\geqslant \frac{1}{k}.
$
Then $a_k a \varphi_k= \varphi_k$ and thus
$
\varphi_kTH_2= a_k a \varphi_kTH_2= a_k\varphi_k a TH_2\subseteq H_2.
$
\end{proof}
Combining Lemma  \ref{lem:Riesz_projector}\ref{enu:Riesz_projection1} and Theorem \ref{thm:Cartan_crossed_products} 
we prove a
 version of Proposition \ref{prop:Riesz_projector}
in the finite type case that works for arbitrary $a\in A$ but gives only forward invariant sets:
\begin{prop}\label{prop:Cartan_Riesz_Projections}
 Suppose that  $aT$, $a\in A\cong C(X)$,  are well presented abstract weighted shifts. Assume the associated 
transfer operator is of finite type and the dual map $\varphi:X\to X$ is topologically free. 
Every Riesz projector $P$ associated to $\sigma(aT)$, $a\in A$,  belongs to $A$ and hence corresponds to a clopen subset of $X$.
Moreover, if the corresponding part of $\sigma(aT)$ does not contain zero, then $P=\mathds{1}_{X_0}\in C(X)$  where $X_0\subseteq X$ is such that $\varphi(X_0)=X_0$, 
$\varphi:X_0\to X_0$ is a homeomorphism and $a$ does not vanish on $X_0$. 
\end{prop}
\begin{proof} 
Let $a\in A$. By Lemma \ref{lem:Riesz_projector}\ref{enu:Riesz_projection1} and Theorem \ref{thm:Cartan_crossed_products}
 every Riesz projector $P$ that corresponds to a part of $\sigma(aT)$  lying in a  disk centered at origin belongs to $A$. 
Therefore $P= \mathds{1}_{X_1}$ and $(1-P)=\mathds{1}_{X_2}$ where $X=X_1\sqcup X_2$. 
Putting $H_2=\mathds{1}_{X_2}H$, we get that $aT\mathds{1}_{X_2}=a\mathds{1}_{X_2}T$ is invertible as an element of $\B(H_2)$. 
By Lemma \ref{lem:Riesz_projector}\ref{enu:Riesz_projection2}, $a\mathds{1}_{X_2}$ is invertible in $\B(H_2)$ as well. 
Hence $\mathds{1}_{X_2}T$ is invertible in $\B(H_2)$ and therefore 
$
T^*\mathds{1}_{X_2}T=L(\mathds{1}_{X_2})
$ is invertible as an element of $C(X_2)\subseteq A=C(X_1)\oplus C(X_2)$. This means that $L(\mathds{1}_{X_2})(y)\neq 0$ if and only if $y\in X_2$, which implies that 
$\varphi(X_2)=X_2$. We claim that the local homeomorphism $\varphi:X_2\to X_2$ is in fact a homeomorphism.
Indeed, the function $h(y):=\varrho(y)|\varphi^{-1}(y)\cap X_2|$, $y\in X_2$, is continuous and strictly positive,
so it is an invertible element in $C(X_2)$. A direct computation shows that $L(h^{-1}\mathds{1}_{X_2})=\mathds{1}_{X_2}$.
The operator $\tilde{T}:=h^{-1/2}\mathds{1}_{X_2}T$ is an invertible isometry in  $\B(H_2)$, because
$
\tilde{T}^*\tilde{T}=T^* h^{-1}\mathds{1}_{X_2}T=L(h^{-1}\mathds{1}_{X_2})=\mathds{1}_{X_2}
$. Let $\tilde{\alpha}:C(X_2)\to C(X_2)$ be the endomorphism given by composition with $\varphi:X_2\to X_2$.
It may be identified with $\mathds{1}_{X_2}\alpha:C(X_2)\to C(X_2)$. Hence for every $b\in C(X_2)$, using (A2) for $(A,T)$, we get
 (A2) for $(C(X_2),\tilde{T})$.
Accordingly, $b\tilde{T}$, $b\in C(X_2)$, are abstract weighted shifts with the unitary $\tilde{T}$.
So $\varphi:X_2\to X_2$ is a homeomorphism by Proposition \ref{prop:T_unitary}.
In particular, $\tilde{T}:=\varrho^{-1/2}\mathds{1}_{X_2}T$ and $\sigma(aT|_{H_2})=\sigma(a\sqrt{\varrho}\tilde{T})$.
Therefore further decomposition of $\sigma(aT|_{H_2})$ using  Riesz projectors gives
 a decomposition of $X_2$ into clopen $\varphi$-invariant subsets by Theorem \ref{thm:lebedev}. 
This implies the assertion (as by Corollary \ref{cor:rotation_invariance}, $\sigma(aT)$  has the circular symmetry).
\end{proof}
Finally, Proposition \ref{prop:Riesz_projector} follows from  Lemma  \ref{lem:Riesz_projector}\ref{enu:Riesz_projection3} and  the  following lemma:
\begin{lem}
Assume that
(A1)-(A3) hold and that $C^*(A,T)$ is canonically isomorphic to $A\rtimes L$. Let $a\in A$. Every Riesz projector $P$ associated to $aT$
belongs to the core subalgebra $B=\clsp\{bT^{n}T^{*n}c: b,c \in A, n\in \N\}$. If in addition 
$T$ commutes with $P$, then $P\in A$
and $P=\mathds{1}_{X_0}\in C(X)$  where $\varphi^{-1}({X_0})={X_0}$.
\end{lem}
\begin{proof} It suffices to consider the case where $\sigma(aT)\cap S^1=\emptyset$
and $ P=\frac{1}{2\pi i}\int_{S^1}(\lambda-aT)^{-1}d\lambda$.
 We exploit the  action $\gamma:\T\to \Aut (C^*(A, T))$ given by Theorem \ref{thm:gauge_uniqueness}.
 For any $z\in S^1$, changing variables in the integral, we get
$$
\gamma_z(P)=\frac{1}{2\pi i}\int_{S^1}(\lambda-zaT)^{-1}d\lambda=\frac{1}{2\pi i}\int_{S^1}z^{-1}\left(\frac{\lambda}z-aT\right)^{-1}d\lambda=P.
$$
Hence $P$ belongs to the fixed point algebra of $\gamma$, which is $B$.
Accordingly,  for every  $\varepsilon >0$ we have $\|P- \sum_{k=0}^n a_kT^{k}T^{*k}b_k\|<\varepsilon$ for some $a_k,b_k\in A$, $k=1,...,n$.
If we assume that $T$ commutes with $P$, then $T^*PT=T^* TP=P$ and therefore
$$
\| P- L(a_0b_0)-\sum_{k=1}^n L(a_k)T^{k-1}T^{*k-1}L(b_k)\|=\|T^*\left(P- \sum_{k=1}^n a_kT^{k}T^{*k}b_k\right)T\|<\varepsilon.
$$
Hence
$\|P- \sum_{k=0}^{n-1} a_k'T^{k}T^{*k}b_k'\|<\varepsilon$ for some $a_k',b_k'\in A$, $k=1,...,n-1$. Proceeding in this way, we find $b\in A$ such that 
$\|P-b\|<\varepsilon$. Thus $P\in A$. Therefore  $P=\mathds{1}_{X_1}\in C(X)$  where ${X_1}\subseteq X$ is  clopen. Put $X_2=X\setminus X_1$. As $T$ commutes with $P$ we have
$L(P)=P$, and this 
implies that $\varphi(X_1)\subseteq X_1$. Since $L(1-P)=1-P$ we also get $\varphi(X_2)\subseteq X_2$. Hence $\varphi^{-1}({X_1})={X_1}$ and $\varphi^{-1}({X_2})={X_2}$.
\end{proof}

\section{The spectrum}\label{sec:spectrum}

Combining  previous results we prove two theorems that describe the spectrum of abstract weighted shift operators $aT$, $a\in A$.
The first one holds when  $a$ is (close to being) invertible and describes $\sigma(aT)$ in terms of a decomposition of $(X,\varphi)$ into closed $\varphi$-invariant sets.
The second one holds for arbitrary $a$  but we need to assume that transfer operator is of finite type, and then the relationship between decomposition of  $\sigma(aT)$ 	and $(X,\varphi)$ is slightly more subtle.
 We end this section with applications to weighted composition operators on $L^2_\mu(\Omega)$.

\begin{thm}[Invertible weight case]\label{thm:main_result1}
 Suppose that  $aT \in \B(H)$, $a\in A\cong C(X)$,  are well presented abstract weighted shifts, i.e. $T$  is an isometry and axioms (A1)-(A3) are satisfied. Assume that the corresponding  map $\varphi:X\to X$  is topologically free.  
Assume also that $a\in A$ is invertible or more generally that $\mathds{1}_{X\setminus a^{-1}(0)}=1$.
Then 
$$
\sigma(aT)=\{z\in\C: |z|\leq r_0\}\cup \bigcup_{r \in R}\{z\in \C: r_{-}\leq |z| \leq  r_+\}
$$
where $ r_0< r_{-}\leq r_{+}$ for  $r\in R$, and there is a 
to  a decomposition  $X=X_0\cup \bigcup_{r\in R} X_{r}$  of $X$ into disjoint  $\varphi$-invariant 
closed sets  and then
$$
r_{-}=\min_{\mu\in \Erg(X_r,\varphi)} \, \exp \int_{X_r} \ln|a|\, d\mu, 
\qquad
r_+=\max_{\mu\in \Erg(X_r,\varphi)} \, \exp \int_{X_r} \ln|a|\, d\mu, 
 $$
for  $r \in R$,
and $r_0=\max_{\mu\in \Inv(X_0,\varphi)}  \exp\left(\int_{{X_0}}\ln\vert a\vert\,d\mu
+\frac{\tau_{L}(\mu)}{2}\right), 
$ where  $\tau_{L}$ is the $t$-entropy for the associated transfer operator $L$. Moreover, $a^{-1}(0)\subseteq X_0$ and for  every $r\in R$, $\varphi:X_r\to X_r$ is a homeomorphism.
 In particular, if $X$ does not contain a non-empty clopen  $\varphi$-invariant set on which $\varphi$ is a homeomorphism and $a$ is non-zero, then  
$$
\sigma(aT)=\left\{z\in \C:|z|\leq \max_{\mu\in \Inv(X,\varphi)}  \exp\left(\int_{{X}}\ln\vert a\vert\,d\mu
+\frac{\tau_{L}(\mu)}{2}\right)\right\}.
$$
\end{thm}
\begin{proof}
 That $\sigma(aT)$ is invariant under rotations over zero follows from 
Corollary \ref{cor:Spectral_consequences}. Hence  $\sigma(aT)$ consists of  a disk $\{z\in\C: |z|\leq r_0\}$ and possibly  a number of annuli $\bigcup_{r \in R}\{z\in \C: r_{-}\leq |z| \leq  r_+\}$.
We may assume that the disk is non-empty, because  otherwise we find ourselves in the situation of Theorem \ref{thm:lebedev},  see Proposition \ref{prop:T_unitary}. 
For simplicity suppose first that $\sigma(aT)$ has exactly two components, so that $
\sigma(aT)=\{z\in\C: |z|\leq r_0\}\cup \{z\in \C: r_{-}\leq |z| \leq  r_+\}
$ where $0\leq r < r_-\leq r_+$.
By Proposition \ref{prop:Riesz_projector}, we have a decomposition $X=X_0\sqcup X_{r}$
into clopen $\varphi$-invariant sets where $\mathds{1}_{X_{0}}$ is the 
Riesz projector for $aT$ corresponding to $\{z\in\C: |z|\leq r_0\}\subseteq \sigma(aT)$.
Put $H_0:=PH$ and $H_{r}=(1-P)H$.  By Lemma \ref{lem:Riesz_projector}, and the assumption that $\mathds{1}_{X\setminus a^{-1}(0)}=1$,
 these spaces are invariant under $T$.  The isometry $T:H_r\to H_r$ is a unitary because both $a|_{H_r}$ and $aT|_{H_{r}}$ are invertible, see Lemma \ref{lem:Riesz_projector}.
In particular, the restrictions $a_0T|_{H_0}$, $a_0\in C(X_0)$, and $a_rT|_{H_r}$, $a_r\in C(X_r)$
are well presented abstract weighted shifts acting on $H_0$ and $H_{r}$, respectively. The associated transfer
operators are restrictions of $L$ to   $C(X_0)$ and $C(X_r)$.
In fact, by Proposition \ref{prop:T_unitary}, $\varphi:X_r\to X_r$ is a homeomorphism and  $L:C(X_r)\to C(X_r)$
is the automorphism given by the composition with  $\varphi^{-1}:X_r\to X_r$.
Now, the formula for $r_0$ follows from Theorem \ref{thm:Antonevich-Bakthin-Lebedev} and formulas for
$r_{-}$ and $r_{+}$ follow from \eqref{eq:spectral_radius_invertible}, because $r_+=r(aT|_{H_r})$ 
and $r_-=r((aT|_{H_r})^{-1})^{-1}=r(a^{-1}\circ \varphi^{-1}T|_{H_r}^{-1})$ and both $cT|_{H_r}$, $c\in C(X_r)$,
 and $cT|_{H_r}^{-1}$, $c\in C(X_r)$, are well presented abstract weighted shifts acting on $H_{r}$ associated with $(X_r, \varphi)$ and $(X_r, \varphi^{-1})$ respectively.

Now let us consider the general case. Let $|\sigma(aT)|:=\{|z|: z\in \sigma(aT)\}=[0, r_0] \cup \bigcup_{r\in R} [r_-,r+]$.
For each $r\in R$ we may find sequences $\{r^-_n\}_{n=1}^\infty$, $\{r^+_n\}_{n=1}^\infty$ lying outside $|\sigma(aT)|$
 such that $r_n^-\nearrow r_-$ and $r_n^+\searrow r_+$. 
Let  $P_{r_n^-,r_n^+}$ be the Riesz projector corresponding to the part of spectrum of $aT$ lying in $\{z\in \C:r_n^- < |z|<r_n^+\}$.
As above, using Proposition \ref{prop:Riesz_projector}, we see that $P_{r_n^-,r_n^+}=\mathds{1}_{X_{r_n^-,r_n^+}}\in C(X)$
for a certain non-empty clopen   $\varphi$-invariant set $X_{r_n^-,r_n^+}$ such that $\varphi:X_{r_n^-,r_n^+}\to X_{r_n^-,r_n^+}$ is a homeomorphism
and $a|_{X_{r_n^-,r_n^+}}\neq 0$.
We put 
 $$
 X_{r}:=\bigcap_{n\in\N} X_{r_n^-,r_n^+}. 
 $$ 
This is a non-empty closed and  $\varphi$-invariant set, which does not 
depend on the choice of the sequences $\{r^-_n\}_{n=1}^\infty$, $\{r^+_n\}_{n=1}^\infty$.
Moreover, $\varphi:X_r\to X_r$ is a homeomorphism. 
Similarly, we define $X_0=\bigcap_{n\in\N} X_{0,r_n}$ where $\{r_n\}_{n=1}^\infty$
is a sequence lying outside  $|\sigma(aT)|$ such that $r_n \searrow r_0$  and 
$\mathds{1}_{X_{0,r_n}}\in C(X)$ is the Riesz projector corresponding to the part of $\sigma(aT)$ lying in $\{z\in \C:|z|< r_n\}$, $n\in \N$.
Again $X_0$ is closed non-empty and $\varphi$-invariant. 

Clearly, the sets $X_r$, $r\in R\cup\{0\}$,  form a decomposition of $X$. To see that $r_0$ is given by
the corresponding variational formula, for each $n\in \N$ choose  $\mu_n\in \Inv(X_{0,r_n},\varphi)$ that maximizes $\max_{\mu\in \Inv(X_{0,r_n},\varphi)}  \exp\left(\int_{{X}_{0,r_n}}\ln\vert a\vert\,d\mu
+\frac{\tau_{L}(\mu)}{2}\right)$. Then 
$$ 
r_0\leq \exp\left(\int_{{X}_{0,r_n}}\ln\vert a\vert\,d\mu_n
+\frac{\tau_{L}(\mu_n)}{2}\right)<r_n.
$$
Let $\mu_{0}\in \Inv(X_{0,r_n},\varphi)$ be a weak$^*$ limit point   of $\{\mu_n\}_{n\in\N}$.
Then the support of $\mu_{0}$ is contained in $X_{0}=\bigcap_{n\in\N} X_{0,r_n}$. 
This and the above inequalities imply that  
$r_0=\exp\left(\int_{{X}_{0}}\ln\vert a\vert\,d\mu_{0}\right)
=\max_{\mu\in \Inv(X_{0},\varphi)}  \exp\left(\int_{{X}_{0}}\ln\vert a\vert\,d\mu\right)$. 
Similarly we get the desired formulas for $r\in R$. Indeed, for each   $n\in \N$ there are  measures $\mu_n^-,\mu_n^+\in \Erg(X_{r_n^-,r_n^+},\varphi) $
such that 
$$
r_{n}^-<  \exp \int_{X_{r_n^-,r_n^+}} \ln|a(x)|\, d\mu_{n}^- \leq r_{-} \leq r_{+} \leq \exp \int_{X_{r_n^-,r_n^+}} \ln|a(x)|\, d\mu_{n}^+ < r_{n}^+.  
 $$
Taking $\mu_0^-$, $\mu_0^+\in  \Inv(X,\varphi)$ to be weak$^*$ limit points of the sequences  $\{\mu_n^-\}_{n\in\N}$, $\{\mu_n^+\}_{n\in\N}$, we see that supports of $\mu_0^-$ and $\mu_0^+$ are contained in  $X_r=\bigcap_{n\in\mathbb N} X_{r_n^-,r_n^+}$, and {\small
$$
r_{-} =\exp \int_{X_{r}} \ln|a(x)|\, d\mu_0^-=\min_{\mu\in \Inv(X_{r},\varphi)}  \exp\int_{{X}_{r}}\ln\vert a\vert\,d\mu=
\min_{\mu\in \Erg(X_{r},\varphi)}  \exp\int_{{X}_{r}}\ln\vert a\vert\,d\mu,
$$
$$
r_+=\exp \int_{X_{r}} \ln|a(x)|\, d\mu_0^+=\max_{\mu\in \Inv(X_{r},\varphi)}  \exp\int_{{X}_{r}}\ln\vert a\vert\,d\mu=
\max_{\mu\in \Erg(X_{r},\varphi)}  \exp\int_{{X}_{r}}\ln\vert a\vert\,d\mu.
$$}
\end{proof}

\begin{thm}[Finite type case]\label{thm:main_result2}
 Suppose that  $aT \in \B(H)$, $a\in A\cong C(X)$,  are well presented abstract weighted shifts such that the associated transfer
operator is of finite type, so that $L(a)(y)=\sum_{x\in\varphi^{-1}(y)}\varrho(x)a(x)$, where $\varrho>0$. Assume  the corresponding  map $\varphi:X\to X$  is topologically free.  
For every $a\in A$ we have
$$
\sigma(aT)=\{z\in\C: |z|\leq r_0\}\cup \bigcup_{r \in R}\{z\in \C: r_{-}\leq |z| \leq  r_+\}
$$
where  $r_0< r_{-}\leq r_{+}$ for  $r\in R$ and Riesz projectors induce  a decomposition  $X=X_0\cup \bigcup_{r\in R} X_{r}$  into disjoint
closed sets where
for each $r\in R$,  $\varphi|_{X_r}:X_r\to X_r$ is a homeomorphism, $a$ does not vanish on $X_r$ and
$$
r_{-}=\min_{\mu\in \Erg(X_r,\varphi)} \, \exp \int_{X_r} \ln(\vert a\vert\sqrt{\varrho })\, d\mu, 
\quad
r_+=\max_{\mu\in \Erg(X_r,\varphi)} \, \exp \int_{X_r} \ln(\vert a\vert\sqrt{\varrho })\, d\mu.
 $$
In particular, if $X$ does not contain a non-empty clopen forward $\varphi$-invariant set  on which $\varphi$ is a homeomorphism and $a$ is non-zero, then  
$$\sigma(aT)=\left\{z\in \C:|z|\leq \max_{\mu\in \Inv(X,\varphi)}  \exp\left(\int_{{X}}\ln\vert a\vert\,d\mu
+\frac{\tau_{L}(\mu)}{2}\right)\right\}.
$$
\end{thm}
\begin{proof}
The proof goes along the same lines as the proof of Theorem \ref{thm:main_result1}, where instead of Proposition \ref{prop:Riesz_projector} we use Proposition \ref{prop:Cartan_Riesz_Projections}. 
In particular, in the proof of Proposition \ref{prop:Cartan_Riesz_Projections} we have seen  that if $X_r\subseteq X$ is a clopen set such that $\mathds{1}_{X_{r}}$ is a  Riesz projector corresponding to a part $D$ of
$\sigma(aT)$ not containing zero, then the operators $b\varrho^{-1/2}T$, $b\in C(X_r)\subseteq A$, can be treated as abstract weighted shifts acting on $H_r:=\mathds{1}_{X_{r}}H$ where 
  $U:=\varrho^{-1/2}T|_{H_r}$ is a unitary. Thus the minimal annulus containing $D=\sigma(aT)\cap D=\sigma(a\varrho^{1/2}U)$ lies between the circles or
radii $
r_{-}=\min_{\mu\in \Erg(X_r,\varphi)} \, \exp \int_{X_r} \ln(\vert a\vert\sqrt{\varrho })\, d\mu, 
$
$
r_+=\max_{\mu\in \Erg(X_r,\varphi)} \, \exp \int_{X_r} \ln(\vert a\vert\sqrt{\varrho })\, d\mu.
 $
\end{proof}
\begin{rem}
The sets $X_r$ in the assertion of Theorem \ref{thm:main_result2} are forward $\varphi$-invariant. If 
it happens that they are backward $\varphi$-invariant, then $\varrho|_{X_r}\equiv 1$ and the  formulas for $r_{-}$
and $r_{+}$ above are the same as in Theorem  \ref{thm:main_result1}. Example \ref{ex:contrexample} shows that 
if $\mathds{1}_{X\setminus a^{-1}(0)}\neq1$, then sets $X_r$ may fail to be backward $\varphi$-invariant.
\end{rem}

\begin{rem} In Theorem \ref{thm:main_result1} and \ref{thm:main_result2} the set $\{z\in\C: |z|\leq r_0\}$ is empty (equivalently $r_0 <0$) if and only if both $a$ 
and $T$ are invertible if and only if $a$ is non-zero everywhere and $\varphi$ is a homeomorphism, cf. Proposition \ref{prop:T_unitary}.  In this case, Theorem \ref{thm:lebedev} also applies.
\end{rem}
\begin{lem}\label{lem:freeness_expansiveness} 
An expanding map $\varphi:X\to X$ on a compact metrizable space $X$ is topologically free if and only if $X$ has no isolated points that are periodic. 
\end{lem}
\begin{proof} If there is a periodic isolated point, then its orbit is a clopen set showing that $\varphi:X\to X$ is not topologically free. 
Conversely, suppose  that $\varphi$ is not topologically free. Then there is a non-empty open set $U\subseteq X$
such that every point $x\in U$ is of period $n>0$ (i.e. $\varphi^n(x)=x$ and $\varphi^k(x)\neq x$ for $k=1,...,n-1$). 
Take any non-empty open set $V$ such that 
$\overline{V} \subseteq U$.  Then $K=\bigcup_{k=0}^{n-1} \varphi^k(\overline{V})$ 
is forward $\varphi$-invariant and $\varphi:K\to K$ is an expansive homeomorphism.  Hence $K$ is finite by  Schwartzman's theorem.
Thus $V$ is open and finite, and therefore it consists of isolated points.
\end{proof}
\begin{cor}\label{cor:spectrum_expansive}
Suppose that  $aT \in \B(H)$, $a\in A\cong C(X)$,  are well presented abstract weighted shifts such that the associated transfer
operator $L:A\to A$ is of finite type. If    the corresponding  map $\varphi:X\to X$
is expanding and $X$ has no isolated points, then 
 $$\sigma(aT)=\sigma_{ess}(aT)=\left\{z\in \C:|z|\leq \max_{\mu\in \Erg(X,\varphi)} \exp\left(\int_{{X}}\ln(\vert a\vert\sqrt{\varrho })\,d\mu
+\frac{h_{\varphi}(\mu)}{2} \right)\right\}
.$$
\end{cor}
\begin{proof}
By Lemma \ref{lem:freeness_expansiveness}  the last part of Theorem \ref{thm:main_result2}, $\sigma(aT)$ is a disk.  By Corollary \ref{cor:Spectral_consequences}, $\sigma(aT)=\sigma_{ess}(aT)$,
  and the formula for $r(aT)$ is given in Corollary \ref{cor:abstract_spectral_radius_expanding}.
\end{proof}

\begin{ex}[Cuntz-Krieger algebras]\label{ex:Cuntz-Krieger algebras}
Let $(\Sigma_{\mathbb{A}}, \sigma_{\mathbb{A}})$ be the one-sided topological Markov shift
 associated to a matrix $\mathbb{A}=[A(i,j)]_{i,j=1}^n$, $n\geq 2$, of zeros and ones, with no zero rows and no zero columns,
see Example \ref{ex:TMS}. The Cuntz-Krieger algebra $\mathcal{O}_{\mathbb{A}}$ associated to $\mathbb{A}$, see \cite{CK}, 
is defined as the universal $C^*$-algebra generated
by a family $\{S_i: i=1,...,n\}$ of partial isometries subject to the following relations
\begin{equation}\label{Cuntz-Krieger algebras}
  S_i^*S_i =\sum_{j=1}^n A(i,j) S_jS_j^*,  \qquad S_i^* S_k=
  \delta_{i,k} S_i^*S_i, \qquad i,k=1 ,...,n,
\end{equation} 
where  $\delta_{i,j}$ is
the Kronecker symbol. The algebra $\mathcal{O}_{\mathbb{A}}$ may be viewed as the crossed product
of $C(\Sigma_{\mathbb{A}})$ by a transfer operator for   $\sigma_{\mathbb{A}}$. More specifically, for
 every  multiindex $\mu = (i_1,\dots,i_k)$, with $i_j \in 1,...,n$,
we  write
\,$S_\mu = S_{i_1}S_{i_2}\dotsm S_{i_k}$.
Then the map $S_\mu S^*_\mu\mapsto  \mathds{1}_{C_{\mu}}$ extends to the isomorphism
$$
A:=\clsp\{S_\mu S^*_\mu :\mu = (i_1,\dots,i_k), \, k=1,\dotsc\}\cong C(\Sigma_{\mathbb{A}}),
$$
given by $S_\mu S^*_\mu\mapsto  \mathds{1}_{C_{\mu}}$.
Also putting $n_j:=\sum_{i=1}^n A(i,j)$ and $P_j:=S_jS_j^*$,
$j=1,...,n$, we have a canonical isometry given by 
$$
T:=\sum_{i,j=1}^n\frac{1}{ \sqrt{n_j}}\,\, S_i P_j,
$$
see \cite{exel3}, \cite{bialowieza}. Then $L(a)=TaT^*$, $a\in A$, where $L(a)(y)=\frac{1}{|\sigma_{\mathbb{A}}^{-1}(y)|}\sum_{x\in \sigma_{\mathbb{A}}^{-1}(y)}a(x)$ is the classical transfer operator
 on  $A\cong C(\Sigma_{\mathbb{A}})$.
The isomorphism $A\cong C(\Sigma_{\mathbb{A}})$ extends uniquely to 
an isomorphism $\mathcal{O}_{\mathbb{A}}\cong C(\Sigma_{\mathbb{A}})\rtimes L$ which  sends $T$ to
the universal isometry in $C(\Sigma_{\mathbb{A}})\rtimes L$.
In particular,  by Theorem \ref{thm:gauge_uniqueness} we may always represent $\mathcal{O}_{\mathbb{A}}$ on a Hilbert space $H$, so
that the operators $aT\in \B(H)$, $a\in A\subseteq \B(H)$, are well presented weighted shifts.
 By Corollaries  \ref{cor:rotation_invariance} and \ref{cor:abstract_spectral_radius_expanding}, for  every $a\in A$, $\sigma(aT)$ has the circular symmetry and 
$$
\ln r(aT)=\max_{\mu\in \Erg(\Sigma_{\mathbb{A}},\sigma_{\mathbb{A}})} \int_{\Sigma_{\mathbb{A}}}\ln\vert a\vert\,d\mu
- \frac{1}{2}\ln \prod_{j=1}^n n_j ^{\mu(C_j)}  + \frac{h_{\sigma_{\mathbb{A}}}(\mu)}{2} 
$$
where $C_j$ is the cylinder set of elements that start with $j$.
In other words, putting $|\sigma_{\mathbb{A}}^{-1}|(y):=|\sigma_{\mathbb{A}}^{-1}(y)|$ we get that the spectral logarithm of $aT$ 
is the square root of the
topological pressure for $\sigma_{\mathbb{A}}$ with the potential $|a|^2/|\sigma_{\mathbb{A}}^{-1}|$:
$
\ln r(aT)= \sqrt{P\left(|a|^2/|\sigma_{\mathbb{A}}^{-1}|,\sigma_{\mathbb{A}}\right)}.
$
The $t$-entropy for the canonical transfer operator $L$ is given by   
$
\tau_{L}(\mu)=h_{\varphi}(\mu) - \ln \prod_{j=1}^n n_j ^{\mu(C_j)},
$
see Corollary \ref{cor:t-entropy vs Kolmogorov-Sinai entropy}.  
If $a\in \text{span}\{S_\mu S^*_\mu :\mu = (i_1,\dots,i_k),  k=1,\dotsc\}$ is invertible, then $\ln (|a||\sigma_{\mathbb{A}}^{-1}|^{1/2})$ is H\"older continuous.
Thus if   $\mathbb{A}=\text{diag}(\mathbb{A}_1,...,\mathbb{A}_N)$ is the decomposition into irreducible matrices as in Example \ref{ex:TMS}, then
denoting by $\mu_{j}$ the Gibbs measure for $(\Sigma_{\mathbb{A}_j}, \sigma_{\mathbb{A}_j})$ with potential $\ln (|a||\sigma_{\mathbb{A}_j}^{-1}|^{1/2})$  we have
$$
\ln r(aT)=\max\limits_{j=1,...,N} \int_{\Sigma_{\mathbb{A}_j}}\ln (|a||\sigma_{\mathbb{A}_j}^{-1}|^{1/2})\,d\mu_j +\frac{h_{ \sigma_{\mathbb{A}_j}}(\mu_j)}{2}
,
$$  cf. Corollary \ref{cor:abstract_spectral_radius_expanding}.
In general, the shift map  $\sigma_{\mathbb{A}}$ is topologically free if and only if 
the space $\Sigma_{\mathbb{A}}$ has no isolated points. This is called \emph{condition (I)} in \cite{CK}.
Let us assume it. Then by \cite[Theorem 2.13]{CK} every family $\{S_i: i=1,...,n\}$ of non-zero operators satisfying
\eqref{Cuntz-Krieger algebras} generates an isomorphic copy of  $\mathcal{O}_{\mathbb{A}}$  (this result can be
deduced from Theorem \ref{thm:isomorphism}). By Corollary \ref{cor:spectrum_expansive}, 
under  condition (I), we have 
$
\sigma(aT)=\sigma_{ess}(aT)=\{z\in \C:|z|\leq r(aT)\}.
$
If all entries in $\mathbb{A}$ are ones then $\mathcal{O}_{\mathbb{A}}$ is the Cuntz algebra $\mathcal{O}_{n}$
and the above description gives Example \ref{ex:Cuntz} as a special case. 
The above considerations can  also be generalized to graph $C^*$-algebras, cf. \cite{kwa_inter}, \cite{kwa_Exel}.
\end{ex}
\begin{rem}\label{rem:Cuntz-Krieger_representaion}
If all states $\{1,...,n\}$ are essential (equivalentely $\Omega(\sigma_{\mathbb{A}})=\Sigma_{\mathbb{A}}$), then  Proposition \ref{prop:cocycle_inducing_measures}
gives a representation of  $\mathcal{O}_{\mathbb{A}}$ on a Hilbert space $L^2_\mu(\Sigma_{\mathbb{A}})$ that sends $C(\Sigma_{\mathbb{A}})$ to multiplication operators
and $T$ to the operator of composition with  $\sigma_{\mathbb{A}}$.
This representation is faithful if and only if $\sigma_{\mathbb{A}}$ is topologically free (condition (I) holds).
\end{rem}
\subsection{Weighted shifts associated with measure dynamical system} 
Suppose we are in the situation of subsection \ref{subsection:Weighted shifts}. Hence
$T$ is an isometry on $H=L^2_\mu(\Omega)$ given by \eqref{eq:wso_measurable_isometry}
where $\Phi:\Omega\to \Omega$ is a measurable map with $\mu\circ \Phi^{-1}\preceq \mu$ and the weight \(\rho:\Omega\to \C\) is a \(\Sigma\)-measurable function. 
Let $A\cong L^\infty_\mu(\Omega)$ be the algebra of   operators of multiplication. Then 
$T$ generates the transfer operator $L:A\to A$, $L(a):=T^*aT$, $a\in A$, see Proposition \ref{prop:abstract_wso_from_measure_system}.
The corresponding endomorphism $\alpha:A\to A$ is given by composition with $\Phi$.

By the Gelfand-Naimark theorem, $A\cong C(X)$ for a compact Hausdorff space $X$. 
Since $A=L^\infty_\mu(\Omega)$
is a von Neumann algebra, $X$ is  a Stonean space. Thus it corresponds via the
Stone duality to the complete Boolean algebra of projections in $L^\infty_\mu(\Omega)$. This in turn
corresponds to measurable sets in $\Sigma$ modulo the equivalence relation: $A\sim B \Longleftrightarrow \mu(A\triangle B)=0$.
The map $\Phi:\Omega\to \Omega$ induces a continuous surjective map $\varphi:X\to X$ such that the endomorphism 
$\alpha:C(X)\to C(X)$ is given by composition with $\varphi$.
We characterise  topological freeness of $\varphi$ via  freeness of $\alpha$ in the sense of \cite{Arveson} which translates to a condition for $\Phi$ that
we call \emph{metrical freeness}:
\begin{lem}\label{lem:characterisation_topol_freeness} 
Retain the above assumptions and notation. Let $\Sigma_+$ be the family of measurable sets with  positive measure.
The following conditions are equivalent:
\begin{enumerate}
\item\label{enu:characterisation_topol_freeness1} $\varphi$ is topologically free, that is for each $n>1$ 
the set $\{x\in X:\varphi^n(x)=x\}$ has  empty interior. 

\item\label{enu:characterisation_topol_freeness2} $\alpha$ acts freely \cite{Arveson}, that is for any non-zero projection $p\in A$ with $p \leq \alpha^n(p)$ and $n>1$ there is a  non-zero projection $q\leq p$ 
such that $ \alpha^n(q)\bot q$.
\item\label{enu:characterisation_topol_freeness3}  $\Phi$  is metrically free, that is for every $P\in \Sigma_+$ and $n>1$ with   $\mu(P\setminus \Phi^{-n}(P))=0$   there is $Q\in \Sigma_+$ such that 
$Q\subseteq P$ and $\mu(Q\cap \Phi^{-n}(Q))=0$. 

\end{enumerate}
\end{lem}
\begin{proof} \ref{enu:characterisation_topol_freeness1}$\Rightarrow$\ref{enu:characterisation_topol_freeness2}.
Let $p\in A$ be a non-zero projection  with $p \leq \alpha^n(p)$. 
Then $p=\mathds{1}_{V}\in C(X)$ for a non-empty clopen $V\subseteq X$ with $V\subseteq \varphi^{-n}(V)$.
By  topological freeness of $\varphi$ there is a point $x\in V$ with $\varphi^{n}(x)\neq x$.
Since $X$ is a Stone space we may find a clopen set $U\subseteq V$ such that  $x\in U$
and $U \cap  \varphi^n(U) = \emptyset$. This implies that $\varphi^{-n}(U)\cap U=\emptyset$. Hence for $q:=\mathds{1}_{U}$ we get $q\leq p$ and $\alpha^n(q) q =0$.

\ref{enu:characterisation_topol_freeness2}$\Rightarrow$\ref{enu:characterisation_topol_freeness1}.
Suppose that $\{x\in X:\varphi^n(x)=x\}$ has  a non-empty interior. Then there is a non-empty clopen 
set $V\subseteq \{x\in X:\varphi^n(x)=x\}$. Moreover, for every non-empty clopen subset $U$ of $V$
we have  $\varphi^{-n}(U)\cap U=U\neq \emptyset$. Hence $p:=\mathds{1}_{V}$ is a non-zero projection in $A$ 
such that for every non-zero projection $q\leq p$ 
we have $ \alpha^n(q)q\neq 0$.

Equivalence 
\ref{enu:characterisation_topol_freeness2}$\Leftrightarrow$\ref{enu:characterisation_topol_freeness3},
in view of definition of $\alpha$,  is straightforward.
\end{proof}
We further extend the terminology by saying that a set $\Omega_0\in \Sigma$ is \emph{$\mu$-almost $\Phi$-invariant} if $\mu(\Phi^{-1}(\Omega_0)\triangle \Omega_0)=0$. 
We say that $\Phi$ is \emph{$\mu$-almost bijective on a set $\Omega_0\in \Sigma$} if the map $\Delta\mapsto \Phi^{-1}(\Delta)\cap \Omega_0$ induces an automorphism of $\Sigma_{\Omega_0}/\sim$ where  $\Sigma_{\Omega_0}:=\{\Delta\subseteq \Omega_0: \Delta\in \Sigma\}$. 
\begin{ex}\label{ex:map_freeness_etc}
Let $\Phi:V\to V$ be a surjective map on a set $V$ equipped with the counting measure $\mu$. Then $\Phi$ is metrically free if and only if $\Phi$ has no periodic points, so its graph is a direct sum of leafless and rootless directed trees  \cite{Stochel}. 
A set $\Omega_0\subseteq V$ is $\mu$-almost $\Phi$-invariant if and only if it is  $\Phi$-invariant, i.e. $\Phi^{-1}(\Omega_0)=\Omega_0$. 
The map  $\Phi$ is $\mu$-almost bijective on a set $\Omega_0\subseteq V$ if and only if 
$\Omega_0$ is forward $\Phi$-invariant and $\Phi:\Omega_0\to \Omega_0$ is a bijection. 
The map $\Phi$ is always $\mu$-almost bijective  on some non-empty set. 
Indeed, for any $v\in V$ we may find a sequence $\{v_n\}\in V$ such that $v=\Phi(v_1)$ and 
$v_n=\Phi(v_{n+1})$, for all $n>0$, and then  $\Phi$ restricted to $\Omega_0:=\{v_n: n>0\} \cup \{\Phi^n(v): n 
\geq 0\}$ is bijective. 
However, there might be no  $\Phi$-invariant sets on which $\Phi$ is bijective. 
In Example \ref{ex:contrexample} there are exactly two forward $\Phi$-invariant  sets on which $\Phi$ is  bijective, but none of these sets is (backward) $\Phi$-invariant.    
\end{ex}
\begin{cor}\label{cor:measure_dyna_spectrum}
Assume that $(Th)(\omega)=\rho(\omega) h(\Phi(\omega))$ defines an isometry $T\in \B(L^2_\mu(\Omega))$, 
see Lemma \ref{lem:wso_measurable_isometry}. Assume also that $\rho \neq 0$ $\mu$-almost everywhere, and that $\Phi:\Omega \to \Omega$ is 
 non-singular and metrically free.
Then  
\begin{enumerate}
\item\label{enu:measure_dyna_spectrum1} the spectrum of every weighted composition operator $aT$, $a\in L^\infty_\mu(\Omega)$,
is invariant under rotation around zero;

\item \label{enu:measure_dyna_spectrum2} If the measure $\mu$ is atomless then $\sigma(aT)=\sigma_{ess}(aT)$ for every  $a\in L^\infty_\mu(\Omega)$.

\item \label{enu:measure_dyna_spectrum3} If  $N$ is the largest number such that 
$\Omega=\Omega_0 \cup \bigcup_{i=1}^N \Omega_i$ decomposes into disjoint $\mu$-almost $\Phi$-invariant sets of positive measure, such that 
$\Phi$ is $\mu$-almost bijective on $\Omega_i$, for $i=1,...,N$, then for every $a\in L^\infty_\mu(\Omega)$ which is non-zero $\mu$-almost everywhere $\sigma(aT)$ has 
at most $N+1$ connected components.

\end{enumerate}
\end{cor}
\begin{proof}
By Proposition \ref{prop:abstract_wso_from_measure_system} the operators  $aT$, $a\in L^\infty_\mu(\Omega)$, 
are well presented abstract weighted shifts. By Lemma \ref{lem:characterisation_topol_freeness}  
the corresponding dual map $\varphi:X\to X$ is topologically free.
Hence \ref{enu:measure_dyna_spectrum1} and \ref{enu:measure_dyna_spectrum2} follow from Corollary \ref{cor:Spectral_consequences}, because the algebra $L^\infty_\mu(\Omega)\cong A\subseteq \B(H)$ of multiplication operators contains no non-zero compact operator 
if and only if $\mu$ is atomless. 

 Item \ref{enu:measure_dyna_spectrum3} follows from Theorem \ref{thm:main_result1}. Indeed, in the present setting the condition $\mathds{1}_{X\setminus a^{-1}(0)}=1$
is equivalent to $a\in L^\infty_\mu(\Omega)$ being non-zero $\mu$-almost everywhere. 
Clopen subsets $X_0$ of $X$ are in an one-to-one correspondence with elements $[\Omega_0]$ in $\Sigma/\sim$.
Moreover, $X_0$ is $\varphi$-invariant if and only if $\Omega_0$ is $\mu$-almost $\Phi$-invariant,
and $\varphi$ restricts to a homeomorphism of $X_0$ if and only if $\Phi$ is $\mu$-almost bijective on $\Omega_0$.
\end{proof}

\begin{cor}\label{cor:abstract_wso_from_measure_system1}
Let $\Phi:\Omega\to \Omega$ be a metrically free map on a measure space $(\Omega,\Sigma,\mu)$ which can be decomposed 
 into disjoint sets 
$V_i\in \Sigma$, $i=1,...,n$, such that  $\Delta\mapsto \Phi^{-1}(\Delta)$ induces an isomorphism 
$\Sigma_{V_i}/\sim\cong \Sigma_{\Phi(\Omega_i)}/\sim$, for every $i=1,...,n$. Assume also that  
$\ess\inf(\frac{d\mu\circ \Phi}{d\mu})>0$ and  that $N$ is the largest number such that $\Omega=\Omega_0 \cup \bigcup_{i=1}^N \Omega_i$
decomposes into disjoint sets  such that $\Phi$ is $\mu$-almost bijective on each $\Omega_i$ and $\mu(\Omega_i)>0$, $i=1,...,N$. 
The spectrum of any operator
$$
(aT_{\Phi}h)(\omega)=a(\omega) h(\Phi(\omega)), \qquad  a\in L^\infty_\mu(\Omega),
$$
acting on $L^2_\mu(\Omega)$ is of the form 
$$
\sigma(aT_{\Phi})=\{z\in\C: |z|\leq r_0\}\cup \bigcup_{i=1}^N\{z\in \C: r_{i,-}\leq |z| \leq  r_{i,+}\}
$$
where the unions are not necessarily disjoint. Thus $\sigma(aT_{\Phi})$ has at most $N+1$ components, and 
if there are no sets of positive measure on which $\Phi$ is $\mu$-almost bijective,
then $\sigma(aT_{\Phi})$ is a disk. 
If  $\mu$ is atomless, then  $\sigma(aT_{\Phi})=\sigma_{ess}(aT_{\Phi})$. 
\end{cor}
\begin{proof} In view of Corollary \ref{cor:abstract_wso_from_measure_system},
 operators $aT_{\Phi}$, $a\in  L^\infty_\mu(\Omega)$,
can be viewed as well presented abstract weighted shifts $bT$, $b\in  L^\infty_\mu(\Omega)$,  where   $T:=\sqrt{\frac{d\mu\circ \Phi}{d\mu}}T_{\Phi}$ and
$b=a\cdot \left(\frac{d\mu\circ \Phi}{d\mu}\right)^{-1/2}$. 
As in the proof of Corollary \ref{cor:measure_dyna_spectrum}
we see that the  corresponding dual map $\varphi:X\to X$ is topologically free and $N$ is the largest number 
such that  $X=X_0\cup \bigcup_{i=1}^N X_{i}$ decomposes into disjoint clopen sets such that 
$\varphi:X_i\to X_i$ is a homeomorphism for each $i=1,...,n$ (it corresponds to the decomposition
$\Omega=\Omega_0 \cup \bigcup_{i=1}^N \Omega_i$). 
In addition, the assumption that $\Phi$ induces isomorphisms $\Sigma_{V_i}/\sim\cong \Sigma_{\Phi(V)}/\sim$
for the decomposition $\Omega=V_1\sqcup ... \sqcup V_n$ and that $\ess\inf(\frac{d\mu\circ \Phi}{d\mu})>0$ implies that the associated transfer operator 
$L(a)=T_\Phi^* a\frac{d\mu\circ \Phi}{d\mu} T$, $a\in  A\cong L^\infty_\mu(\Omega)$ is of finite type 
(with a quasi-basis given by $u_i:=\left(\frac{d\mu\circ \Phi}{d\mu}\right)^{-1/2} \mathds{1}_{\Omega_i} $, $i=1,...,n$).
Thus the assertion follows from the last part of Theorem \ref{thm:main_result2} (and Corollary \ref{cor:measure_dyna_spectrum}).
\end{proof}
\begin{ex}
Suppose that $\Phi(z)=z^n$, $n>1$, is the map defined on $\Omega=S^1=\{z\in \C:|z|=1\}$
and let $H=L^2_\mu(S^1)$ where  $\mu$ is the normalized Lebesgue measure on $S^1$. Consider the operators $aT\in \B(H)$, $a\in L^\infty_\mu(S^1)$, given by
$aTh (z)=a(z) h(\varphi(z))$, $h\in H$. By Corollary \ref{cor:abstract_wso_from_measure_system1} we have
$
\sigma(aT)=\sigma_{ess}(aT)=\left\{z\in\C: |z|\leq r(aT)\right\}.
$
 If $a\in C(S^1)$ we may deduce the same  from  Corollary \ref{cor:spectrum_expansive}. 
\end{ex}

\begin{ex}\label{ex:directed_trees}
Let $\Phi:V\to V$ be a surjective map on a countable set $V$, which is minimal in the sense that
there are no non-trivial $\Phi$-invariant subsets of $V$. Then $\Phi$ corresponds to a  
leafless and rootless directed tree $\mathfrak{T}$, see  \cite{Stochel}.
Assume that this tree has a finite number of branches.  Then all the bounded weighted shifts on $\mathfrak{T}$
are of the form $(aT_{\Phi}h)(v)=a(v) h(\Phi(v))$,  $a\in \ell^\infty(V)$, 
cf. Remark \ref{rem:weighted_shifts_on_trees}.  Assumptions of  Corollary \ref{cor:abstract_wso_from_measure_system1}
are satisfied with $N=1$, cf. Example \ref{ex:map_freeness_etc}. 
Thus for every weighted shift $S_\lambda$   on the directed tree $\mathfrak{T}$ with  bounded weights $\lambda=\{\lambda_v\}_{v\in V}$ 
the spectrum $\sigma(S_\lambda)$ has at most two connected components. In fact it may have two connected components by
Example \ref{ex:contrexample}. If all weights are non-zero, then  $\sigma(S_\lambda)$ is connected by Corollary \ref{cor:measure_dyna_spectrum}.
\end{ex}

\end{document}